\documentclass[11pt,reqno, a4paper]{amsart}
\usepackage{amsmath,amsfonts,amsthm,amsopn,color,amssymb,enumitem}
\usepackage{palatino}
\usepackage{graphicx}
\usepackage{caption}
\usepackage{subcaption}
\usepackage[colorlinks=true]{hyperref}
\usepackage{geometry}
\hypersetup{urlcolor=blue, citecolor=red, linkcolor=blue}

\usepackage{cite}
\usepackage{esint}
\usepackage{pgfplots}
\usepackage{mathrsfs}
\usetikzlibrary{arrows}
\usepackage{verbatim}
\usepackage{mathtools}

\newtheorem{theorem}{Theorem}[section]
\newtheorem{lemma}[theorem]{Lemma}

\newtheorem{remark}[theorem]{Remark}

\newtheorem{corollary}[theorem]{Corollary}

\newcommand{\R}{\mathbb{R}}

\newcommand{\D}{\mathbb{D}}

\newcommand{\N}{\mathbb{N}}

\def\bbm[#1]{\mbox{\boldmath $#1$}}
\newcommand{\beq }{\begin{equation}}
\newcommand{\eeq }{\end{equation}}

\def\sideremark#1{\ifvmode\leavevmode\fi\vadjust{\vbox to0pt{\vss% the remark3
 \hbox to 0pt{\hskip\hsize\hskip1em%                          will appear only
 \vbox{\hsize3cm\tiny\raggedright\pretolerance10000%          on the side
  \noindent #1\hfill}\hss}\vbox to8pt{\vfil}\vss}}}%

\pgfplotsset{compat=1.18}

\begin{document}

\title[Global second order optimal regularity]{Global second order optimal regularity for the vectorial $p$-Laplacian}

\author{Berardino Sciunzi, Giuseppe Spadaro, Domenico Vuono}

\email[Berardino Sciunzi]{sciunzi@mat.unical.it}
\email[Giuseppe Spadaro]{giuseppe.spadaro@unical.it}
\email[Domenico Vuono]{domenico.vuono@unical.it}
\address[B. Sciunzi, G. Spadaro, D. Vuono]{Dipartimento di Matematica e Informatica, Università della Calabria,
Ponte Pietro Bucci 31B, 87036 Arcavacata di Rende, Cosenza, Italy}



%\maketitle

\begin{abstract}
We obtain optimal regularity  results for solutions to vectorial $p$-Laplace equations  $$ -{\boldsymbol \Delta}_p{\boldsymbol u}=-\operatorname{\bf div}(|D{\boldsymbol u}|^{p-2}D{\boldsymbol u}) = {\boldsymbol f}(x)\,\, \mbox{ in $\Omega$}\,.$$
More precisely we address the issue of global second order estimates for the stress field.
\end{abstract}

\maketitle

\section{Introduction}

\noindent In this work we deal with second order regularity estimates for solutions to vectorial $p$-Laplace equations. More precisely we shall consider problems involving the vectorial operator $-{\boldsymbol \Delta}_p{\boldsymbol u}$ defined, for  smooth functions, by 
\beq\label{pLaplace}
-{\boldsymbol \Delta}_p{\boldsymbol u}=-\operatorname{\bf div}(|D{\boldsymbol u}|^{p-2}D{\boldsymbol u})
\eeq
where $p>1$ and $D{\boldsymbol u}$ is the Jacobian of  the  vector field ${\boldsymbol u}\,:\, \Omega \rightarrow \mathbb{R}^N$, with $N\geq 2$ 
and 
 $\Omega$ is a bounded Lipschitz domain  in $\mathbb{R}^n$, with $n\geq 2$. We shall use the notation  ${\boldsymbol u}=(u^1,\ldots,u^N)$, for a weak 
 solution to the $p$-Laplace system 
\beq\label{system1}%\tag{$p$-$N$}
-\operatorname{\bf div}(|D{\boldsymbol u}|^{p-2}D{\boldsymbol u})
=   {\boldsymbol f}(x)\quad \text{in } \Omega,
\eeq
where ${\boldsymbol f}:\Omega \to\mathbb{R}^N$ is given.

\noindent Our main focus is on proving optimal regularity of the stress field $|D\boldsymbol u|^{p-2} D \boldsymbol u$ up to the boundary. To this aim, we consider either Dirichlet
\begin{equation}\label{Dir_cond}\tag{D}
    \boldsymbol u = 0 \quad \text{on } \partial \Omega,
\end{equation}
or Neumann homogeneous boundary conditions,
\begin{equation}\label{Neu_cond}\tag{N}
    {\frac{\partial \boldsymbol u}{\partial \boldsymbol{\nu}}}= 0 \quad \text{on } \partial \Omega,
\end{equation}
where $\boldsymbol\nu$ denotes the outward unit vector on $\partial \Omega$.

\noindent This problem has been studied by Cianchi and Maz’ya in \cite{Cma} where the authors proved that if $\boldsymbol{f} \in L^2(\Omega)$, then $|D \boldsymbol u|^{p-2} D \boldsymbol u \in W^{1,2}(\Omega)$ under minimal assumptions on $\partial \Omega$, when either Dirichlet or Neumann boundary conditions are prescribed. This result can be seen as an optimal achievement since it shows that the stress field is in $W^{1,2}(\Omega)$ if and only if the source term is in $L^2(\Omega)$. Our point of view is to prove stronger second-order regularity of the solutions, obviously under stronger assumptions on the source term and on the regularity of the domain. In particular, a suitable assumption on the domain consists in assuming some integrability properties of the weak curvatures of $\partial \Omega$. One can request that the functions that locally describe the boundary of $\Omega$ are endowed with second-order weak derivatives which belong to a specific Lorentz-Zygmund space depending on the dimension $n$ (see Section 2 for details).
Luckily we succeed in obtaining sharp estimates, in the sense of \cite[Remark 1.4]{SMM_vec}, namely:

\vspace{0.3 cm}
%\textcolor{red}{dire qualcosa su spazi lorentz e zugmund vedere antonini giusto per introdurre il teorema}
\begin{theorem}\label{teo1INTRO}
	Let $\Omega$ be a bounded Lipschitz domain in $\R^n$, with $\partial \Omega\in W^2X$, where

\begin{equation}\label{X}
    X=
    \begin{cases}
        L^{n-1,1} & \mbox{if $n \geq 3$,}\\
        L\log{L} & \mbox{if $n =2$.}
    \end{cases}
\end{equation}

    Let $\boldsymbol{u}$ be a weak solution to either the Neumann problem \eqref{system1} + \eqref{Neu_cond} or the Dirichlet problem \eqref{system1} + \eqref{Dir_cond}, with $$\boldsymbol f(x)\in W^{1,1}(\Omega)\cap L^q(\Omega), \quad \text{with} \quad q>n.$$\\
    If $1 <
    p< \frac{3}{2}$, let us assume that
    \begin{equation}\label{h_p}
    	 \alpha < h(p) \vcentcolon = \frac{(2-p)(p+2\sqrt{2(p-1)}+1)}{(p-3)^2},
    \end{equation}
    while $\alpha < 1$ if $p\geq \frac{3}{2}$.

Then,
\begin{equation}\label{ciao2INTRO}
    |D \boldsymbol u|^{\gamma-1} D \boldsymbol u \in W^{1,2}(\Omega),
\end{equation}
for any  $\gamma \geq \frac{p-\alpha}{2}$.

\end{theorem}

\begin{remark}
    Notice that with $\gamma = p-1$ we recover the result in \cite{Cma}, namely $|D \boldsymbol u|^{p-2} D \boldsymbol u \in W^{1,2}(\Omega)$. Note also that our result, when $\gamma = p-1$, agrees with the estimate obtained in \cite{BaCiDiMa}, namely $|D \boldsymbol{u}|^{p-2} D \boldsymbol{u} \in W^{1,2}(\Omega)$, for $p > 2(2-\sqrt{2})$. 
\end{remark}

\begin{remark}
    
\end{remark}

As a consequence of Theorem \ref{teo1INTRO} we have the following 

\begin{corollary}
   Let $\boldsymbol u$, $\Omega$ and $\boldsymbol f$ as in Theorem \ref{teo1INTRO}. If
 $1<p < 3$, we have that $$\boldsymbol{u} \in W^{2,2}(\Omega).$$
\end{corollary}

We can prove the same results for bounded convex domains, with no additional regularity on $\partial \Omega$. This is due to the fact that the integrals over $\partial \Omega$ depend on its curvatures. In particular, when $\Omega$ is convex, these integrals could be disregarded since they have a definite sign.

%\textcolor{red}{frasi con domini convessi metterte qualcosa tipo come fa antonini}

\begin{theorem}[Convex domains]\label{conv_d}
    Let $\Omega$ be a bounded convex open set in $\mathbb{R}^n$. Let $\alpha$, $\boldsymbol{u}$ and $\boldsymbol{f}$ as in Theorem \ref{teo1INTRO}. Then $|D \boldsymbol u|^{\gamma-1} D \boldsymbol u \in W^{1,2}(\Omega)$, for any $\gamma \geq \frac{p-\alpha}{2}$. Moreover, if $1<p<3$, we have that $u\in W^{2,2}(\Omega)$.
\end{theorem}

\noindent In \cite{SMM_vec}, the authors proved sharp second order estimates up to the boundary for the stress field of the $p$-Laplace equation, namely proving a global version of the local results in \cite{DamSci}. 
Note that the result in \cite{SMM_vec} was obtained for $C^3$-smooth domains 
via a fine argument based on the Fermi coordinates.

\noindent Here we succeed in proving global estimates for the vectorial case and, at the same time, we consider more general domains, namely Lipschitz domains which boundary are prescribed by a suitable condition involving integrability properties of the second fundamental form $\mathcal{B}$ associated to the manifold $\partial \Omega$. \\

\noindent Our results provide a global counterpart to the local estimates recently obtained in \cite{SMM_vec} but the techniques are completely new. Actually we   introduce some new ideas
also borrowing arguments from \cite{Cma} and \cite{SMM_vec}.\\

\noindent We point out that Theorem \ref{teo1INTRO} holds without any sign assumption on the source term $\boldsymbol{f}$. If, in addition, a sign condition is imposed on $\boldsymbol{f}$, then we obtain the integrability properties of $|D\boldsymbol{u}|^{-1}$, namely
\begin{theorem}\label{peso_stima_Intro}
     Let $\boldsymbol u$, $\Omega$ and $\boldsymbol f$ as in Theorem \ref{teo1INTRO}. Suppose that, for any $x\in \overline \Omega$, there exists $\rho_x$ such that:
     \begin{equation*}
         f^\beta \geq \tau_x > 0\quad \text{or } f^\beta \leq -\tau_x < 0 \quad \text{in } B_\rho(x)\cap \overline\Omega,
     \end{equation*}
  for some $\beta =1,...,N$. Then:
     \begin{equation*}
         \int_{\Omega} \frac{1}{{|D \boldsymbol{u}|}^\sigma}\leq C,
     \end{equation*}
     for any
     \begin{equation}\label{cond_sigma_I}
         \sigma <
         \begin{cases}
             p-2+h(p) & \mbox{if $p\leq \frac{3}{2}$}\\
             p-1 & \mbox{if $p>\frac{3}{2}$},
         \end{cases}
     \end{equation}
     and $C=C(\tau,\sigma,p,n,N,L_\Omega,d_\Omega,\|\boldsymbol{f}\|_{L^q(\Omega)},\|\boldsymbol{f}\|_{W^{1,1}(\Omega)})$ is a positive constant. \\
     \noindent Consequently, for $p \geq 3$, we have that
     \begin{center}
     $\boldsymbol{u}\in W^{2,q}(\Omega),\quad$ for any $ \,\,1\leq q < \frac{p-1}{p-2}$.
     \end{center} 
\end{theorem}
\begin{remark}
    The sign condition on the source term in Theorem \ref{peso_stima_Intro}  for the Neumann case, actually requires $f^\beta$ to have a sign on the boundary. This is not required in the proofs but we have in mind the compatibility   condition for the existence of the solution. For the Dirichlet case, since a general Hopf boundary Lemma is not available for the vectorial case, the sign condition turns out to be not so restrictive even if the source term has a sign all over the domain. 
\end{remark}
In light of Theorem \ref{conv_d}, if $\Omega$ is a convex domain we can prove the following result with no additional assumptions on the boundary.

%\textcolor{red}{qualcosa sui convex qui}
\begin{theorem}[Convex domains]\label{conv_d_peso}
    Let $\Omega$ be a bounded convex open set in $\mathbb{R}^n$. Let $\alpha$, $\boldsymbol{u}$ and $\boldsymbol{f}$ as in Theorem \ref{peso_stima_Intro}. Then \begin{equation*}
         \int_{\Omega} \frac{1}{{|D \boldsymbol{u}|}^\sigma}\leq C.
     \end{equation*}
    Moreover, if $p\geq 3$, we have that $u\in W^{2,q}(\Omega)$, with $q< (p-1)/(p-2)$
\end{theorem}

\noindent 

\noindent Before starting the proofs of our results let us discuss the state of the art on the regularity theory for $p$-Laplace equations. The $C^{1,\alpha}$ theory is quite well understood in the scalar case, we refer to \cite{AvKuMi,Bar,12,13,23,24,DeFMi,DM,CiDeSci,DM2,DM3,DM4,GuMo,Libb,LU} and the references therein. Second order estimates in the scalar case goes back to \cite{antonuovo,DamSci,Ce,ACCFM,CiaMa,CaRiSc,SMM_vec,tolk,lou}. Some new estimates on the third derivatives of the solution were recently obtained in \cite{beni}.

\noindent On the other hand, the vectorial case is very much harder to study. Founding results have been obtained in \cite{ChenDiB,Cma,Cma2,BaCiDiMa,MMSV,KuuMin,16,DP,Marc,DiKaSc2,M}.

\vspace{0.3cm}

Lastly, for completeness, we include a local estimate:

\begin{theorem}[Local case]\label{local}
  Let $\Omega$ be an open set in $\R^n$, 
    and let $\boldsymbol{u}$ be a weak solution to \eqref{system1},  with $\boldsymbol f(x)\in W^{1,1}_{loc}(\Omega)\cap L^q_{loc}(\Omega)$ and $q>n$.
    If $1 <
    p< \frac{3}{2}$, let us assume that $\alpha < h(p)$, with $h(p)$ given in \eqref{h_p}, while $\alpha < 1$ if $p\geq \frac{3}{2}$.

Then,
$$
    |D \boldsymbol u|^{\gamma-1} D \boldsymbol u \in W^{1,2}_{loc}(\Omega),
$$
for any $\gamma \geq \frac{p-\alpha}{2}$. Moreover, we have that$$\boldsymbol{u}\in W_{loc}^{2,2}(\Omega)\quad \text{for } 1<p<3.$$\\
    On the other hand, for $\tilde\Omega\subset\subset\hat\Omega\subset\subset\Omega$,
    if we suppose that, for any $x\in  \hat\Omega$, there exists $\rho_x$ such that:
     \begin{equation*}
         f^\beta \geq \tau_x > 0\quad \text{or } f^\beta \leq -\tau_x < 0 \quad \text{in } B_\rho(x)\cap \hat \Omega,
     \end{equation*}
  for some $\beta =1,...,N$,  it follows that 
     \begin{equation*}
         \int_{\tilde \Omega} \frac{1}{{|D \boldsymbol{u}|}^\sigma}\leq C,
     \end{equation*}
     for any $\sigma$ satisfying \eqref{cond_sigma_I}, and where $C=C(\tau,\sigma,p,n,N, \tilde\Omega,\hat\Omega,\|\boldsymbol{f}\|_{L^q(\hat \Omega)},\|\boldsymbol{f}\|_{W^{1,1}(\hat \Omega)})$ is a positive constant. \\
     \noindent Consequently, for $p \geq 3$, we prove that
     \begin{center}
     $\boldsymbol{u}\in W^{2,q}(\tilde\Omega),\quad$ for any $ \,\,1\leq q < \frac{p-1}{p-2}$.
     \end{center} 
\end{theorem}
\begin{remark}
    Since the technique used to prove Theorem \ref{local} is similar to that of Theorem \ref{teo1INTRO} and Theorem \ref{peso_stima_Intro}, we omit the proof of Theorem \ref{local}. In the case of $\boldsymbol{f}\in W_{loc}^{1,1}(\Omega)\cap C_{loc}^{0,\beta}(\Omega)$ and $\alpha < p-1$ for $p< 2$ a local estimate is proved in \cite{MMSV}.
\end{remark}

We organize the paper as follows. In Section \ref{seconda}, we state some useful preliminary results. In Section \ref{SOE}, we deal with second order estimates for solutions of both \eqref{pLaplace} + \eqref{Dir_cond} and \eqref{pLaplace} + \eqref{Neu_cond}. In particular we prove Theorem \ref{teo1INTRO}. Finally, in Section \ref{integrability_section}, with the additional assumption that the right hand side has a sign, we prove Theorem \ref{peso_stima_Intro}.
%\textcolor{red}{dire qualcosa che facciamo nelle varie sezioni}

\section{Notation and preliminary results}\label{seconda}
\subsection*{Notation}\label{Notazioni}
Generic fixed numerical constants will be denoted by $C$ (with subscript in some cases) and will be allowed to vary within a single line or formula. We also denote with $|A|$ the Lebesgue measure of the set $A$.

We will use the bold style to stress the vectorial nature of  different quantities. For instance, a $N$-vectorial function $\boldsymbol{w}$ defined in $\Omega$ will be written as
$$
\boldsymbol{w}(x)=\left(w^1(x),\ldots,w^N(x) \right),
$$
where $w^\ell$ is a scalar function defined in $\Omega$ for $\ell=1,\ldots,N$.

\

We say that a vector field ${\boldsymbol u}$ is a weak solution to the Neumann problem \eqref{system1} + \eqref{Neu_cond} if and only if ${\boldsymbol u} \in W^{1,p}(\Omega):=W^{1,p}(\Omega;\mathbb{R}^N)$ and 
\begin{equation*}%\label{weakN}
\int_{\Omega} |D{\boldsymbol u}|^{p-2}D {\boldsymbol u} : D {\boldsymbol \psi}\, dx = \int_{\Omega} \langle{\boldsymbol f}, \boldsymbol{\psi}\rangle\, dx, \quad \forall  {\boldsymbol \psi} \in W^{1,p}(\Omega).
\end{equation*}
Obviously, a necessary condition for existence of weak solutions of the Neumann problem relies on the following compatibility constraint:
\begin{equation*}
    \int_{\Omega} f^j\, dx = 0, \quad \forall{j=1,...,N}. 
\end{equation*}

In analogy, a vector field ${\boldsymbol u}$ weakly  solves the Dirichlet problem \eqref{system1} + \eqref{Dir_cond} if and only if ${\boldsymbol u} \in W^{1,p}_0(\Omega):=W^{1,p}_0(\Omega;\mathbb{R}^N)$ and 
\beq\label{weak}
\int_{\Omega} |D{\boldsymbol u}|^{p-2}D {\boldsymbol u} : D {\boldsymbol \psi}\, dx = \int_{\Omega} \langle{\boldsymbol f}, \boldsymbol{\psi}\rangle\, dx, \quad \forall  {\boldsymbol \psi} \in W^{1,p}_0(\Omega).
\eeq

Here, $D \boldsymbol{u}$ denotes the Jacobian matrix:
\[D {\boldsymbol u}=  \begin{pmatrix}
         \nabla u^1 \\
         \vdots \\
         \nabla u^N
        \end{pmatrix},\] 
        where 
\[\nabla u^\ell = \left(\frac{\partial u^\ell}{\partial x_1},\ldots,\frac{\partial u^\ell}{\partial x_n}\right)\,\, \text{for } \ell=1,\ldots,N,\]
and
\[\displaystyle{|D {\boldsymbol u}|=\sqrt{\sum_{\ell=1}^N\sum_{j=1}^n \left(\frac{\partial u^\ell}{\partial x_j}\right)^2}}.\] 
We also use the notations $\displaystyle u_j^\ell=\dfrac{\partial u^\ell}{\partial x_j}$ and $\displaystyle u_{ij}^\ell=\dfrac{\partial^2 u^\ell}{\partial x_i\partial x_j}$.

We point out that along this paper the symbol $\,:\,$ stands for the scalar product of the matrices rows, namely 
$$
\mathcal{M}:\mathcal{N}=\sum_{i=1}^{q} \mathcal{M}^i \cdot \mathcal{N}^i=\sum_{i=1}^{q} \sum_{j=1}^{r} \mathcal{M}^i_j \mathcal{N}^i_j,
$$
where $\mathcal{M},\mathcal{N} \in \mathbb{R}^{q \times r}$ are real matrices, whereas $ \cdot$ denotes the scalar product of two real vectors.

Let $\{ \boldsymbol{e}^\alpha\}_{\alpha =1}^N$, and ${\{ \boldsymbol{e}_j\}}^n_{j =1}$ be the canonical basis of $\mathbb{R}^N$ and $\mathbb{R}^n$, respectively. We will also denote a second-order tensor of size $N\times n$ as
\begin{equation*}
    \eta = \eta_j^\alpha \ \boldsymbol{e}^\alpha \otimes \boldsymbol{e}_j,
\end{equation*}
and a third-order tensor of size $N\times n \times n$ as
\begin{equation*}
    \xi = \xi_{ij}^\alpha \ \boldsymbol{e}^\alpha \otimes \boldsymbol{e}_j \otimes \boldsymbol{e}_i.
\end{equation*}
In particular, by $D^2\boldsymbol{u}$ we mean:
\begin{equation*}
    D^2 \boldsymbol{u} = u_{ij}^\alpha \ \boldsymbol{e}^\alpha \otimes \boldsymbol{e}_j \otimes\boldsymbol{e}_i,
\end{equation*}
and we define its norm as follows
\[\|D^2 {\boldsymbol u} \|=\sqrt{\sum_{i=1}^N\sum_{j,k}(u^i_{jk})^2}.\]

In order to simplify the computations below, we define the vector
\begin{equation}\label{eq:vect}
D^2\boldsymbol{u}D\boldsymbol{u}:=\left(\begin{array}{c}D\boldsymbol{u}:D\boldsymbol{u}_{1}\\
\vdots \\D\boldsymbol{u}:D\boldsymbol{u}_{i}\\\vdots
\\
D\boldsymbol{u}:D\boldsymbol{u}_{n}
\end{array}\right)
\end{equation}
The previously defined vector stands for the application of a third-order tensor to a second-order one.

In what follows we often use the inequality
\begin{equation}\label{eq:FraCo}
|D^2\boldsymbol u D\boldsymbol u|\leq |D\boldsymbol{u}|\|D^2\boldsymbol{u}\|.
\end{equation}

\

\subsection*{Setting of the problem}
Before start this section, let us stress our gratitude for the tools that we found in \cite{Ant,ACCFM,BaCiDiMa,Cma2,CiaMa,Cma,Gr}.\\
We consider $\Omega$ a bounded Lipschitz domain and we assume that the functions of $(n-1)$ variables that
locally describe the boundary of $\Omega$ are twice weakly differentiable with second derivatives in a suitable space $X$, briefly $\partial \Omega \in  W^2X$.
An open set $\Omega$ in $\R^n$ is called a Lipschitz domain if there exist constants $L_\Omega>0$ and $R_\Omega\in (0,1)$ such that, for every $x_0\in \partial \Omega$ and $R \in (0, R_\Omega]$ there exist an orthogonal coordinate system centered at $0\in \R^n$ and an $L_\Omega$-Lipschitz continuous function $\psi : B'_{R}\rightarrow (-l,l)$, where $B'_{R}$ denotes the ball in $R^{n-1}$, centered at $0' \in \R^{n-1}$ and with radius $R$, and $l=R(1+L_\Omega)$, satisfying 

\begin{equation*}
    \begin{split}
        \partial \Omega \cap \left(B'_R \times (-l,l)\right)=\{(x',\psi(x')) : x'\in B'_R\} \\ \Omega \cap \left(B'_R \times (-l,l)\right)=\{(x',x_n) : \psi(x')<x_n<l\}.
    \end{split}
\end{equation*}

 Moreover, we call Lipschitz characteristic of $\Omega$ the couple $\mathcal{L}_\Omega=(L_\Omega,R_\Omega)$. 

\

\

Let $\Omega$ be a Lipschitz domain with Lipschitz characteristic $\mathcal{L}_\Omega = (L_\Omega, R_\Omega)$, and let $\rho\in L^1(\partial \Omega)$ be a nonnegative function. We set, for $r \in(0, R_\Omega]$,

\begin{equation}\label{defK}
    \mathcal{K}_{\Omega,\rho}(r)=\sup_{E\subset B_r(x) ,  x\in\partial \Omega} \frac{\int_{\partial \Omega\cap E}\rho \,d\mathcal{H}^{n-1}}{cap(B_r(x),E)},
\end{equation}
and
\begin{equation}\label{integrability}
    \Psi_{\Omega,\rho}(r)=
    \begin{cases}
        \sup_{x \in \partial \Omega} \| \rho \|_{L^{n-1,\infty}(\partial \Omega \cap B_r(x))} & \mbox{if $n \geq 3$}\\
        \sup_{x \in \partial \Omega} \| \rho \|_{L^{1,\infty}\log{L}(\partial \Omega \cap B_r(x))} & \mbox{if $n = 2$}
    \end{cases}
\end{equation}
\begin{comment}
\begin{equation}\label{integrability_n}
    \Psi_{\Omega,\rho}(r)=
        \sup_{x \in \partial \Omega} \| \rho \|_{\tilde X(\partial \Omega \cap B_r(x))}
\end{equation}
\end{comment}

Here, $B_r(x)$ stands for the ball centered at $x$, with radius $r$, the notation $cap(B_r(x),E)$ denotes the capacity of the set $E$ relative to the ball $B_r(x)$, and $\mathcal{H}^{n-1}$ is the $(n-1)$-dimensional Hausdorff measure. 

Recall that the Lorentz spaces $L^{p,q}(\partial \Omega)$, with $p \geq 1$, are the Banach function spaces endowed with the norm:
\begin{equation*}
    \| f \|_{L^{p,q}(\partial \Omega)} = 
    \begin{cases}
        \left \{ \int_0^{\mathcal{H}^{n-1}(\partial\Omega)} [t^\frac{1}{p} f^{**}(t)]^q \frac{dt}{t}\right\}^\frac{1}{q}& \mbox{if $1 \leq q < \infty$}\\
        \sup_{0<t<\mathcal{H}^{n-1}(\partial\Omega)} \{t^\frac{1}{p} f^{**}(t)\}& \mbox{if $q=\infty$}
    \end{cases}
\end{equation*}
for a measurable function $f$ on $\partial \Omega$. Here, $f^{**}(t) = \frac{1}{t} \int_{0}^t f^*(s)ds$ for $t>0$, where $f^*$ stands for the decreasing rearrangement of $f$. The Lorentz-Zygmund space $L^{p,q}\log L(\partial \Omega)$, with $p\geq 1$, is equipped with the norm
\begin{equation*}
    \| f \|_{L^{p,q}\log L(\partial \Omega)} = 
    \begin{cases}
        \left \{ \int_0^{\mathcal{H}^{n-1}(\partial\Omega)} [t^\frac{1}{p} \log(1+\frac{1}{t}) f^{**}(t)]^q \frac{dt}{t}\right\}^\frac{1}{q}& \mbox{if $1 \leq q < \infty$}\\
        \sup_{0<t<\mathcal{H}^{n-1}(\partial\Omega)} \{t^\frac{1}{p} \log(1+\frac{1}{t}) f^{**}(t)\}& \mbox{if $q=\infty$}.
    \end{cases}
\end{equation*}

Moreover, we set 
\begin{equation}\label{KB}
    \mathcal{K}_{\Omega}(r)=\sup_{E\subset B_r(x) ,  x\in\partial \Omega} \frac{\int_{\partial \Omega\cap E}|\mathcal{B}| \,d\mathcal{H}^{n-1}}{cap(B_r(x),E)},
\end{equation}
and
\begin{equation}\label{integrabilityB}
    \Psi_{\Omega}(r)=\sup_{x \in \partial \Omega} \| \mathcal{B}\|_{\tilde X(\partial \Omega \cap B_r(x))}
\end{equation}
where $\mathcal{B}$ stands for the weak second fundamental form on $\partial \Omega$, $|\mathcal{B}|$ its norm and
\begin{equation*}
    \tilde X=
    \begin{cases}
        L^{n-1,\infty} & \mbox{if $n \geq 3$,}\\
        L^{1,\infty}\log{L} & \mbox{if $n =2$.}
    \end{cases}
\end{equation*}

\begin{comment}
\begin{equation*}
    X=
    \begin{cases}
        L^{n-1,1} & \mbox{if $n \geq 3$,}\\
        L \log{L} & \mbox{if $n =2$.}
    \end{cases}
\end{equation*}
\end{comment}

\begin{remark}\label{antony2}
    Notice that if $\partial \Omega \in W^2 X$, where $X$ is defined as in \eqref{X}, then the following condition holds:
    \begin{equation}\label{cond_B}
        \lim_{r \rightarrow 0^+} \left (\sup_{x \in \partial \Omega} \| \mathcal{B} \|_{\tilde X(\partial \Omega \cap B_r(x))}\right) < c,
    \end{equation}
    for a suitable positive constant $c=c(n,N,p,L_\Omega,d_\Omega)$. Indeed,
    \begin{equation*}
        L^{n-1,1} \subset L^{n-1} \subset L^{n-1,\infty},
    \end{equation*}
    and
    \begin{equation*}
        L \log L \subset L^{1,\infty}\log L.
    \end{equation*}
    In particular, condition \eqref{cond_B} is fulfilled if $\partial \Omega \in C^2$.
\end{remark}

Now we state two lemmas which will be useful to us later, see \cite[Section 6]{ACCFM} and \cite{Cma,Ant}. 
\begin{lemma}\label{antoninho2}
    Let $\Omega$ be a bounded Lipschitz domain in $\R^n$ with Lipschitz characteristic $\mathcal{L}_\Omega =(L_\Omega, R_\Omega)$. Assume that $\rho$ is a nonnegative function on $\partial \Omega$ such that  $\rho\in L^1(\partial \Omega)$. Then,
    \begin{equation}
        \mathcal{K}_{\Omega,\rho}(r) \leq c \Psi_{\Omega,\rho}(r),
    \end{equation}
    for some constant $c=c(n,L_\Omega)$, for every $r \in (0,R_\Omega]$.
\end{lemma}

\begin{lemma}\label{antoninho}
    Let $\Omega$ be a bounded Lipschitz domain in $\R^n$, with Lipschitz characteristic $\mathcal{L}_\Omega =(L_\Omega, R_\Omega).$ Assume that $\rho$ is a nonnegative function on $\partial \Omega$ such that  $\rho\in L^1(\partial \Omega)$. Then, 
    \begin{equation}\label{antony}
        \int_{\partial \Omega \cap B_r(x_0)}v^2\rho \,d\mathcal{H}^{n-1} \leq 32(1+L_\Omega)^4K_{\Omega,\rho}(r)\int_{\Omega\cap B_r(x_0)}|\nabla v|^2 \,dx
    \end{equation}
    for every $x_0\in \partial \Omega$, for every $r\in (0,R_\Omega]$, and for every $v \in W^{1,2}_0(B_r(x_0))$.
\end{lemma}
We conclude this section with the following result (see \cite[Lemma 3.5]{BaCiDiMa}):

\begin{lemma}\label{algebra}
        Let $N \geq 2$, $0 \leq \theta \leq \frac{1}{2}$, $\sigma \in \mathbb{R}$ such that $\theta + \sigma \geq 1$. Assume that the vectors $\omega ^ \beta \in \mathbb{R}^n$ and the matrices $H^\beta \in \mathbb{R}^{n \times n}_{sym}$, with $\beta = 1, \dots, N$, satisfy the following:
    \begin{equation}\label{constr_omega}
        \sum_{\beta = 1}^N {|\omega ^\beta|}^2 \leq 1,
    \end{equation}
    \begin{equation}\label{constr_H}
        \sum_{\beta = 1}^N {|H ^\beta|}^2 \leq 1.
    \end{equation}
    Let,
    \begin{equation*}
        J = \left |\sum_{\beta = 1}^N H^\beta \omega ^\beta \right|^2 , \quad J_0= \sum_{\beta = 1}^N  \left \langle\omega ^\beta , \sum_{l=1}^N H^l \omega^l \right \rangle^2 , \quad J_1 = \sum_{\beta=1}^N {|H^\beta|}^2.
    \end{equation*}
    Then
    \begin{equation*}
        J - \theta J_0 - \sigma J_1 \leq 
         \begin{cases}
        0 & \mbox{if $\theta \in [0,\frac{1}{3}]$,}\\
        \max  \left \{ 0,\frac{(\theta+1)^2}{8\theta}-\sigma \right\} & \mbox{if $\theta \in (\frac{1}{3},\frac{1}{2}]$.}\\
         \end{cases}
    \end{equation*}
\end{lemma}

\section{Second order estimates}\label{SOE}

A key tool of our approach relies on global integral inequalities for smooth functions and smooth domains:
\begin{theorem}\label{stimaC^2}
    Let $\Omega$ be a bounded open set in $\R^n$, with $\partial \Omega \in C^2$. If $1<p< \frac{3}{2}$, let us assume $0\leq \alpha <h(p)$, with $h(p)$ as in \eqref{h_p}, and if $p\geq \frac{3}{2}$, let  $\alpha\in [0,1)$. There exists a constant $\overline c=\overline c(n,N,p,\alpha,L_\Omega)$ such that, if
    \begin{equation}\label{condKr}
         K_\Omega (r) \leq K(r), \quad \forall{} r \in (0,1),
    \end{equation}
    for some $K$ satisfying
    \begin{equation}\label{condizionesuk}
        \lim_{r\rightarrow 0^+} K (r) <\overline c,
    \end{equation}
    then
    \begin{equation}\label{abcdef}
    \begin{split}
        &\int_\Omega(\varepsilon +|D\boldsymbol{u}|^2)^{\frac{p-2-\alpha}{2}} \|D^2\boldsymbol{u}\|^2 \, dx \\ & \leq \mathcal{C}\left(\int_\Omega{(\varepsilon +|D\boldsymbol{u}|^2)^{\frac{p-\alpha}{2}}} \, dx\right.  \\&  \qquad\left.+\sum_{i=1}^n\int_\Omega \langle \operatorname{\bf div}( (\varepsilon+|D{\boldsymbol u}|^2)^{\frac{p-2}{2}}D{\boldsymbol u}), \partial _{x_i}\left(\frac{\boldsymbol{u}_i}{(\varepsilon+|D{\boldsymbol u}|^2)^{\frac{\alpha}{2}}}\right)\rangle\,dx\right),
    \end{split}
\end{equation}
if condition \eqref{Neu_cond} is in force.\\
On the other hand, if condition \eqref{Dir_cond} is assigned we obtain:
\begin{equation}\label{abcdefDir}
    \begin{split}
        &\int_\Omega(\varepsilon +|D\boldsymbol{u}|^2)^{\frac{p-2-\alpha}{2}} \|D^2\boldsymbol{u}\|^2 \, dx \\ & \leq \mathcal{C}\left(\int_\Omega{(\varepsilon +|D\boldsymbol{u}|^2)^{\frac{p-\alpha}{2}}} \, dx\right.  \\&  \qquad\left.+\sum_{i=1}^n\int_\Omega \langle \operatorname{\bf div}( (\varepsilon+|D{\boldsymbol u}|^2)^{\frac{p-2}{2}}D{\boldsymbol u}), \partial _{x_i}\left(\frac{\boldsymbol{u}_i}{(\varepsilon+|D{\boldsymbol u}|^2)^{\frac{\alpha}{2}}}\right)\rangle\,dx\right.\\
        &\qquad\left. +\int_\Omega |\operatorname{\bf div}( (\varepsilon+|D{\boldsymbol u}|^2)^{\frac{p-2}{2}}D{\boldsymbol u})| |D{\boldsymbol u}|^{1-\alpha} \,dx \right),
    \end{split}
\end{equation}
for every vector field $\boldsymbol{u}\in C^3(\Omega)\cap C^2(\overline \Omega)$ and for some positive constant $\mathcal{C}(n,N,p,\alpha,L_\Omega,d_\Omega,K)$, where $d_\Omega$ is the diameter of $\Omega$. 

In particular, if $\Omega$ is convex, inequalities \eqref{abcdef} and \eqref{abcdefDir} hold with the constant $\mathcal{C}$ only depending on $n,N,p,\alpha,L_\Omega,d_\Omega$.
\end{theorem}

\begin{proof}
   
    Let $\varepsilon\in (0,1)$. Since $\boldsymbol u\in C^3(\Omega)$, we deduce that 
    \begin{equation}\label{a_1}
    \begin{split}
        &-\langle \partial _{x_i}\left(\operatorname{\bf div}( (\varepsilon+|D{\boldsymbol u}|^2)^{\frac{p-2}{2}}D{\boldsymbol u})\right), \boldsymbol{u}_i\rangle\\ &=-\langle\operatorname{\bf div}( (\varepsilon+|D{\boldsymbol u}|^2)^{\frac{p-2}{2}}D{\boldsymbol u}_i+(p-2)(\varepsilon+|D{\boldsymbol u}|^2)^{\frac{p-4}{2}}(D\boldsymbol u_i:D\boldsymbol u)D\boldsymbol u),\boldsymbol{u}_i\rangle.
        \end{split}
\end{equation}
Let $\varphi \in C_c^\infty(\mathbb{R}^n)$, multiplying \eqref{a_1} by $\varphi ^2/(\varepsilon+|D{\boldsymbol u}|^2)^{\frac{\alpha}{2}}$ and integrating both sides of the previous equality over $\Omega$, yields to:

\begin{equation}\label{a_2}
    \begin{split}
        -&\int_\Omega \langle \partial _{x_i}\left(\operatorname{\bf div}( (\varepsilon+|D{\boldsymbol u}|^2)^{\frac{p-2}{2}}D{\boldsymbol u})\right), \boldsymbol{u}_i\rangle\frac{\varphi ^2}{(\varepsilon+|D{\boldsymbol u}|^2)^{\frac{\alpha}{2}}}\,dx\\ &=-\int_\Omega\langle\operatorname{\bf div}( (\varepsilon+|D{\boldsymbol u}|^2)^{\frac{p-2}{2}}D{\boldsymbol u}_i),\boldsymbol u_i\rangle\frac{\varphi ^2}{(\varepsilon+|D{\boldsymbol u}|^2)^{\frac{\alpha}{2}}}\,dx\\ &\quad-(p-2)\int_\Omega \operatorname{\bf div}((\varepsilon+|D{\boldsymbol u}|^2)^{\frac{p-4}{2}}(D\boldsymbol u_i:D\boldsymbol u)D\boldsymbol u),\boldsymbol{u}_i\rangle\frac{\varphi ^2}{(\varepsilon+|D{\boldsymbol u}|^2)^{\frac{\alpha}{2}}}\,dx.
    \end{split}
\end{equation}

Since $\boldsymbol u \in C^2(\overline \Omega)$, integrating by parts both integrals in the right hand side of \eqref{a_2}, we obtain: 

\begin{equation}\label{a_3}
    \begin{split}
        -&\int_\Omega \langle \partial _{x_i}\left(\operatorname{\bf div}( (\varepsilon+|D{\boldsymbol u}|^2)^{\frac{p-2}{2}}D{\boldsymbol u})\right), \boldsymbol{u}_i\rangle\frac{\varphi ^2}{(\varepsilon+|D{\boldsymbol u}|^2)^{\frac{\alpha}{2}}}\,dx\\ &=\int_\Omega (\varepsilon+|D{\boldsymbol u}|^2)^{\frac{p-2}{2}}\left( D{\boldsymbol u}_i:D\left(\frac{\boldsymbol u_i\varphi ^2}{(\varepsilon+|D{\boldsymbol u}|^2)^{\frac{\alpha}{2}}}\right)\right)\,dx\\
        &\quad+(p-2)\int_\Omega (\varepsilon+|D{\boldsymbol u}|^2)^{\frac{p-4}{2}}(D\boldsymbol u_i:D\boldsymbol u) \left( D\boldsymbol u:D\left(\frac{\boldsymbol{u}_i\varphi ^2}{(\varepsilon+|D{\boldsymbol u}|^2)^{\frac{\alpha}{2}}}\right)\right)\,dx\\
        &\quad-(p-2)\int_{\partial \Omega}(\varepsilon+|D{\boldsymbol u}|^2)^{\frac{p-4-\alpha}{2}}(D\boldsymbol u_i:D\boldsymbol u)\langle (D\boldsymbol{u})^T\boldsymbol{u}_i,\boldsymbol{\nu}\rangle\varphi^2\,d\mathcal{H}^{n-1}\\
        &\quad-\int_{\partial \Omega}(\varepsilon+|D{\boldsymbol u}|^2)^{\frac{p-2-\alpha}{2}}\langle (D\boldsymbol{u}_i)^T\boldsymbol{u}_i,\boldsymbol{\nu}\rangle\varphi^2\,d\mathcal{H}^{n-1},
    \end{split}
\end{equation}
where $\boldsymbol\nu$ stands for the outward unit vector on $\partial \Omega$.

Let us focus on the left hand side of \eqref{a_3}. An integration by parts shows:
\begin{equation*}
    \begin{split}
        -&\int_\Omega \langle \partial _{x_i}\left(\operatorname{\bf div}( (\varepsilon+|D{\boldsymbol u}|^2)^{\frac{p-2}{2}}D{\boldsymbol u})\right), \boldsymbol{u}_i\rangle\frac{\varphi ^2}{(\varepsilon+|D{\boldsymbol u}|^2)^{\frac{\alpha}{2}}}\,dx\\
        &=\int_\Omega \langle \operatorname{\bf div}( (\varepsilon+|D{\boldsymbol u}|^2)^{\frac{p-2}{2}}D{\boldsymbol u}),\partial _{x_i} \left(\frac{ \boldsymbol{u}_i \varphi ^2}{(\varepsilon+|D{\boldsymbol u}|^2)^{\frac{\alpha}{2}}} \right) \rangle\,dx\\
&\quad-\int_{\partial \Omega}\langle \operatorname{\bf div}( (\varepsilon+|D{\boldsymbol u}|^2)^{\frac{p-2}{2}}D{\boldsymbol u}),\boldsymbol{u}_i\rangle \nu_i \frac{\varphi ^2}{(\varepsilon+|D{\boldsymbol u}|^2)^{\frac{\alpha}{2}}}\,d\mathcal{H}^{n-1},
    \end{split}
\end{equation*}
where $\nu_i$ represents the $i$-th component of the unit normal vector $\boldsymbol{\nu}$.

Summing up, using the previous equality, we can rewrite \eqref{a_3} as follows:
\begin{equation}\label{a_intermedio}
    \begin{split}
        &\int_\Omega (\varepsilon+|D{\boldsymbol u}|^2)^{\frac{p-2}{2}}\left( D{\boldsymbol u}_i:D\left(\frac{\boldsymbol u_i\varphi ^2}{(\varepsilon+|D{\boldsymbol u}|^2)^{\frac{\alpha}{2}}}\right)\right)\,dx\\
        &\quad+(p-2)\int_\Omega (\varepsilon+|D{\boldsymbol u}|^2)^{\frac{p-4}{2}}(D\boldsymbol u_i:D\boldsymbol u) \left( D\boldsymbol u:D\left(\frac{\boldsymbol{u}_i\varphi ^2}{(\varepsilon+|D{\boldsymbol u}|^2)^{\frac{\alpha}{2}}}\right)\right)\,dx\\
        &=(p-2)\int_{\partial \Omega}(\varepsilon+|D{\boldsymbol u}|^2)^{\frac{p-4-\alpha}{2}}(D\boldsymbol u_i:D\boldsymbol u)\langle (D\boldsymbol{u})^T\boldsymbol{u}_i,\boldsymbol{\nu}\rangle\varphi^2\,d\mathcal{H}^{n-1}\\
        &\quad+\int_{\partial \Omega}(\varepsilon+|D{\boldsymbol u}|^2)^{\frac{p-2-\alpha}{2}}\langle (D\boldsymbol{u}_i)^T\boldsymbol{u}_i,\boldsymbol{\nu}\rangle\varphi^2\,d\mathcal{H}^{n-1}\\
        &\quad+\int_\Omega \langle \operatorname{\bf div}( (\varepsilon+|D{\boldsymbol u}|^2)^{\frac{p-2}{2}}D{\boldsymbol u}),\partial _{x_i} \left(\frac{ \boldsymbol{u}_i \varphi ^2}{(\varepsilon+|D{\boldsymbol u}|^2)^{\frac{\alpha}{2}}} \right) \rangle\,dx\\
&\quad-\int_{\partial \Omega}\langle \operatorname{\bf div}( (\varepsilon+|D{\boldsymbol u}|^2)^{\frac{p-2}{2}}D{\boldsymbol u}),\boldsymbol{u}_i\rangle \nu_i \frac{\varphi ^2}{(\varepsilon+|D{\boldsymbol u}|^2)^{\frac{\alpha}{2}}}\,d\mathcal{H}^{n-1}.
    \end{split}
\end{equation}

Computations show that:
\begin{equation}\label{D_test}
\begin{split}
    D\left(\frac{\boldsymbol u_i\varphi ^2}{(\varepsilon+|D{\boldsymbol u}|^2)^{\frac{\alpha}{2}}}\right) = &D\boldsymbol u_i \frac{\varphi ^2}{(\varepsilon+|D{\boldsymbol u}|^2)^{\frac{\alpha}{2}}} + 2\varphi \frac{\boldsymbol{u}_{i}(\nabla \varphi)^T}{(\varepsilon+|D{\boldsymbol u}|^2)^{\frac{\alpha}{2}}}\\
    &\quad- \alpha \varphi^2 \frac{\boldsymbol{u}_{i}\left(D^2\boldsymbol u D\boldsymbol u\right)^T}{(\varepsilon+|D{\boldsymbol u}|^2)^{\frac{\alpha+2}{2}}}
\end{split}
\end{equation}

Finally, plugging \eqref{D_test} in \eqref{a_intermedio}, we deduce:

\begin{equation}\label{a_4}
    \begin{split}
        \int_\Omega&\frac{(\varepsilon +|D\boldsymbol{u}|^2)^{\frac{p-2}{2}}}{(\varepsilon +|D\boldsymbol{u}|^{2})^{\frac{\alpha}{2}}} |D\boldsymbol{u}_{i}|^2\varphi^2 \, dx\\
&-\alpha \int_\Omega\frac{(\varepsilon +|D\boldsymbol{u}|^2)^{\frac{p-2}{2}}}{(\varepsilon +|D\boldsymbol{u}|^2)^{\frac{\alpha+2}{2}}}(D\boldsymbol{u}_{i}:\boldsymbol{u}_{i}\left(D^2\boldsymbol u D\boldsymbol u\right)^T)\varphi^2\, dx\\
&+(p-2)\int_\Omega \frac{(\varepsilon +|D\boldsymbol{u}|^2)^{\frac{p-4}{2}}}{(\varepsilon +|D\boldsymbol{u}|^{2})^{\frac{\alpha}{2}}}(D\boldsymbol{u}:D\boldsymbol{u}_{i})^2\varphi^2\, dx\\&-\alpha(p-2) \int_\Omega\frac{(\varepsilon +|D\boldsymbol{u}|^2)^{\frac{p-4}{2}}}{(\varepsilon +|D\boldsymbol{u}|^2)^{\frac{\alpha+2}{2}}}(D\boldsymbol{u}:D\boldsymbol{u}_{i})(D\boldsymbol{u}:\boldsymbol{u}_{i}\left(D^2\boldsymbol u D\boldsymbol u\right)^T)\varphi^2\, dx\\
=&-2\int_\Omega\frac{(\varepsilon +|D\boldsymbol{u}|^2)^{\frac{p-2}{2}}}{(\varepsilon +|D\boldsymbol{u}|^{2})^{\frac{\alpha}{2}}}(D\boldsymbol{u}_{i}: \boldsymbol{u}_{i}(\nabla \varphi)^T) \varphi \  dx\\
&-2(p-2)\int_\Omega\frac{(\varepsilon +|D\boldsymbol{u}|^2)^{\frac{p-4}{2}}}{(\varepsilon +|D\boldsymbol{u}|^{2})^{\frac{\alpha}{2}}}(D\boldsymbol{u}:D\boldsymbol{u}_{i})(D\boldsymbol{u}: \boldsymbol{u}_{i}(\nabla \varphi)^T)\varphi \, dx\\
&+(p-2)\int_{\partial \Omega}(\varepsilon+|D{\boldsymbol u}|^2)^{\frac{p-4-\alpha}{2}}(D\boldsymbol u_i:D\boldsymbol u)\langle (D\boldsymbol{u})^T\boldsymbol{u}_i,\boldsymbol{\nu}\rangle\varphi^2\,d\mathcal{H}^{n-1}\\
&+\int_{\partial \Omega}(\varepsilon+|D{\boldsymbol u}|^2)^{\frac{p-2-\alpha}{2}}\langle (D\boldsymbol{u}_i)^T\boldsymbol{u}_i,\boldsymbol{\nu}\rangle\varphi^2\,d\mathcal{H}^{n-1}\\
&+\int_\Omega \langle \operatorname{\bf div}( (\varepsilon+|D{\boldsymbol u}|^2)^{\frac{p-2}{2}}D{\boldsymbol u}),\partial _{x_i} \left(\frac{ \boldsymbol{u}_i \varphi ^2}{(\varepsilon+|D{\boldsymbol u}|^2)^{\frac{\alpha}{2}}} \right) \rangle\,dx\\
&-\int_{\partial \Omega}\langle \operatorname{\bf div}( (\varepsilon+|D{\boldsymbol u}|^2)^{\frac{p-2}{2}}D{\boldsymbol u}),\boldsymbol{u}_i\rangle \nu_i \frac{\varphi ^2}{(\varepsilon+|D{\boldsymbol u}|^2)^{\frac{\alpha}{2}}}\,d\mathcal{H}^{n-1}=:\mathcal{I}_1+\cdots+\mathcal{I}_6.
    \end{split}
\end{equation}

Now we estimate the left-hand side of \eqref{a_4}. Exploiting a computation found in \cite{M}, we deduce
\begin{eqnarray}\label{terminesinistra}
&&\sum_{i=1}^{n}\left\{ \int_\Omega\frac{(\varepsilon +|D\boldsymbol{u}|^2)^{\frac{p-2}{2}}}{(\varepsilon +|D\boldsymbol{u}|^{2})^{\frac{\alpha}{2}}} |D\boldsymbol{u}_{i}|^2\varphi^2 \, dx\right.
\\\nonumber
&&\quad-\alpha \int_\Omega\frac{(\varepsilon +|D\boldsymbol{u}|^2)^{\frac{p-2}{2}}}{(\varepsilon +|D\boldsymbol{u}|^2)^{\frac{\alpha+2}{2}}}(D\boldsymbol{u}_{i}:\boldsymbol{u}_{i}\left(D^2\boldsymbol u D\boldsymbol u\right)^T)\varphi^2\, dx\\\nonumber
&&\quad+(p-2)\int_\Omega \frac{(\varepsilon +|D\boldsymbol{u}|^2)^{\frac{p-4}{2}}}{(\varepsilon +|D\boldsymbol{u}|^{2})^{\frac{\alpha}{2}}}(D\boldsymbol{u}:D\boldsymbol{u}_{i})^2\varphi^2\, dx\\\nonumber
\\\nonumber
&&\quad\left.-\alpha(p-2) \int_\Omega\frac{(\varepsilon +|D\boldsymbol{u}|^2)^{\frac{p-4}{2}}}{(\varepsilon +|D\boldsymbol{u}|^2)^{\frac{\alpha+2}{2}}}(D\boldsymbol{u}:D\boldsymbol{u}_{i})(D\boldsymbol{u}:\boldsymbol{u}_{i}(D^2\boldsymbol u D\boldsymbol u)^T)\varphi^2\, dx\right\}\\\nonumber
&&=\int_\Omega(\varepsilon +|D\boldsymbol{u}|^2)^{\frac{p-2-\alpha}{2}} \|D^2\boldsymbol{u}\|^2\varphi^2 \, dx\\\nonumber
&&\quad+(p-2-\alpha)\int_\Omega(\varepsilon +|D\boldsymbol{u}|^2)^{\frac{p-4-\alpha}{2}} \sum_{i=1}^n(D\boldsymbol{u}:D\boldsymbol{u}_{i})^2\varphi^2 \, dx\\\nonumber
&&\quad-\alpha(p-2)\int_\Omega(\varepsilon +|D\boldsymbol{u}|^2)^{\frac{p-6-\alpha}{2}}\sum_{k=1}^N\langle D^2\boldsymbol u D\boldsymbol u, \nabla u^k\rangle \langle D^2\boldsymbol u D\boldsymbol u, \nabla u^k\rangle\varphi^2 \, dx. 
\end{eqnarray}
We distinguish two cases:\\
If $1<p\leq 2$, we rewrite the right hand side of \eqref{terminesinistra} as follows
\begin{equation*}
    \int_\Omega(\varepsilon +|D\boldsymbol{u}|^2)^{\frac{p-2-\alpha}{2}} \mathcal{I}\varphi^2 \, dx,
\end{equation*}
where
\begin{equation}\label{I}
\begin{split}
    \mathcal{I}& \vcentcolon = \|D^2\boldsymbol{u}\|^2 + (p-2-\alpha) \frac{\sum_{i=1}^n(D\boldsymbol{u}:D\boldsymbol{u}_{i})^2}{(\varepsilon +|D\boldsymbol{u}|^2)} - \alpha(p-2) \frac{\sum_{k=1}^N\langle D^2\boldsymbol u D\boldsymbol u, \nabla u^k\rangle^2}{(\varepsilon +|D\boldsymbol{u}|^2)^2}\\
    &=\|D^2\boldsymbol{u}\|^2 + (p-2-\alpha) \frac{|D^2\boldsymbol{u}D\boldsymbol{u}|^2}{(\varepsilon +|D\boldsymbol{u}|^2)} - \alpha(p-2) \frac{ |D\boldsymbol{u}(D^2\boldsymbol{u}D\boldsymbol{u})|^2}{(\varepsilon +|D\boldsymbol{u}|^2)^2}
\end{split}
\end{equation}

We claim that:

\begin{equation}\label{K_N}
        \mathcal{I} \geq  K_N(p,\alpha) \|D^2\boldsymbol{u}\|^2,
\end{equation}
where
\begin{equation}\label{defK_n}
    K_N(p,\alpha) = 
    \begin{cases}
        1+\frac{(\alpha(p-2)+p-2-\alpha)^2}{8\alpha(p-2)} & \mbox{if $p \in (1,\frac{3}{2})$, $\alpha \in(\frac{p-2}{3p-5},1)$,}\\
        (p-1)(1-\alpha) & \mbox{if $p \in (1,\frac{3}{2})$, $\alpha \in [0,\frac{p-2}{3p-5}]$,}\\
        (p-1)(1-\alpha) & \mbox{if $p \in [\frac{3}{2},2)$.}
    \end{cases}
\end{equation}

In order to prove the claim, we call for an application of Lemma \ref{algebra}. To this aim, we define
\begin{equation*}
    \omega ^\beta = \frac{\nabla {u}^\beta}{(\varepsilon+|D\boldsymbol{u}|^2)^{\frac{1}{2}}} \in \mathbb{R}^n, \qquad H^\beta = \nabla^2 {u} ^\beta \in \mathbb{R}^{n \times n}_{sym},
\end{equation*}
for $\beta = 1,...,N$, where $\nabla ^2 u^\beta$ stands for the Hessian matrix associated to the $\beta$-th component of $\boldsymbol{u}$.\\
It is easy to see that assumption \eqref{constr_omega} is satisfied. Furthermore, if $|D^2 \boldsymbol{u} (x)| > 1$ at some point $x$, we can normalize $D^2 \boldsymbol{u} (x)$ so that condition \eqref{constr_H} holds.

Computations lead to
\begin{equation*}
        J = \frac{|D^2\boldsymbol{u}D\boldsymbol{u}|^2}{(\varepsilon+|D\boldsymbol{u}|^2)}, \quad J_0= \frac{|D\boldsymbol{u}(D^2\boldsymbol{u}D\boldsymbol{u})|^2}{(\varepsilon+|D\boldsymbol{u}|^2)^2} , \quad J_1 = {\|D^2\boldsymbol{u}\|}^2.
\end{equation*}
In this way we can rewrite \eqref{I} as follows:
\begin{equation*}
    \mathcal{I}=J_1 + (p-2-\alpha) J - \alpha(p-2) J_0
\end{equation*}
Next, set
\begin{equation*}
    \theta=\frac{\alpha(p-2)}{p-2-\alpha}.
\end{equation*}

 Notice that $\theta \in [0,\frac{1}{2})$, for every $\alpha \in [0,1)$. In addition, if $p \in [\frac{3}{2},2]$, then $\theta \in [0,\frac{1}{3})$, for any  $\alpha \in [0,1)$. On the other hand if $p \in (1,\frac{3}{2})$, then $\theta \in [0,\frac{1}{3}]$ if $\alpha \leq \frac{p-2}{3p-5}$ and $\theta \in (\frac{1}{3},\frac{1}{2})$ if $\alpha > \frac{p-2}{3p-5}$.

Consequently, let us choose
\begin{equation*}
    \sigma = 
    \begin{cases}
       \frac{(\theta+1)^2}{8\theta} & \mbox{if $p \in (1,\frac{3}{2})$, $\alpha \in(\frac{p-2}{3p-5},1)$,}\\
       \frac{p-2-\alpha(p-1)}{p-2-\alpha} & \mbox{if $p \in (1,\frac{3}{2})$, $\alpha \in [0,\frac{p-2}{3p-5}]$,}\\
       \frac{p-2-\alpha(p-1)}{p-2-\alpha} & \mbox{if $p \in [\frac{3}{2},2]$.}
        
    \end{cases}
\end{equation*}
In particular, with this choice, $\theta + \sigma \geq 1$. Thus, a direct application of Lemma \ref{algebra} leads to
\begin{equation*}
    J \leq \frac{\alpha(p-2)}{p-2-\alpha} J_0 +\sigma J_1.
\end{equation*}
As a consequence,
\begin{equation*}\label{K_N1}
        \mathcal{I} \geq (1+\sigma(p-2-\alpha))\|D^2\boldsymbol{u}\|^2= K_N(p,\alpha) \|D^2\boldsymbol{u}\|^2,
\end{equation*}
with $K_{N}(p,\alpha)$ as in \eqref{defK_n}.

In particular, by computations, it follows that $K_N(p,\alpha)>0$ if
\begin{equation*}
    \alpha < h(p) \vcentcolon = \frac{(2-p)(p+2\sqrt{2(p-1)}+1)}{(p-3)^2},
\end{equation*}
when $p\in (1,\frac{3}{2})$. We remark that if $p\in (1,\frac 32)$, then $h(p)>\frac{p-2}{3p-5}$. Thus, $h(p)$ is the best bound for $\alpha$.

\vspace{0.5 cm}

On the other hand, if $p>2$, by \eqref{eq:FraCo}, we note that 
\begin{eqnarray}\label{eq:FraCo2}
&&\sum_{k=1}^N\langle D^2\boldsymbol u D\boldsymbol u, \nabla u^k\rangle \langle D^2\boldsymbol u D\boldsymbol u, \nabla u^k\rangle=\sum_{k=1}^N\langle D^2\boldsymbol u D\boldsymbol u, \nabla u^k\rangle^2\\\nonumber
&&\qquad=|D\boldsymbol{u}(D^2\boldsymbol u D\boldsymbol u)|^2\leq |D\boldsymbol{u}|^2|D^2\boldsymbol u D\boldsymbol u|^2,
\end{eqnarray}
where
\begin{equation*}
    |D^2\boldsymbol u D\boldsymbol u|^2=\sum_{i=1}^n(D\boldsymbol{u}:D\boldsymbol{u}_{i})^2.
\end{equation*}
Thus, by \eqref{eq:FraCo2}, the right-hand side of equality \eqref{terminesinistra} becomes:
\begin{equation*}
    \begin{split}
        &\int_\Omega(\varepsilon +|D\boldsymbol{u}|^2)^{\frac{p-2-\alpha}{2}} \|D^2\boldsymbol{u}\|^2\varphi^2 \, dx\\
&\qquad+(p-2-\alpha)\int_\Omega(\varepsilon +|D\boldsymbol{u}|^2)^{\frac{p-4-\alpha}{2}} \sum_{i=1}^n(D\boldsymbol{u}:D\boldsymbol{u}_{i})^2\varphi^2 \, dx\\
&\qquad-\alpha(p-2)\int_\Omega(\varepsilon +|D\boldsymbol{u}|^2)^{\frac{p-6-\alpha}{2}}\sum_{k=1}^N\langle D^2\boldsymbol u D\boldsymbol u, \nabla u^k\rangle \langle D^2\boldsymbol u D\boldsymbol u, \nabla u^k\rangle\varphi^2 \, dx.\\
&\quad\geq \int_\Omega(\varepsilon +|D\boldsymbol{u}|^2)^{\frac{p-2-\alpha}{2}} \|D^2\boldsymbol{u}\|^2\varphi^2 \, dx\\
&\qquad+ (p-2-\alpha-\alpha(p-2)) \int_\Omega(\varepsilon +|D\boldsymbol{u}|^2)^{\frac{p-4-\alpha}{2}} |D^2\boldsymbol u D\boldsymbol u|^2\varphi^2 \, dx
\end{split}
\end{equation*}
If $p-2-\alpha-\alpha(p-2) \geq 0$, we have:
\begin{equation*}
    \begin{split}
        &\int_\Omega(\varepsilon +|D\boldsymbol{u}|^2)^{\frac{p-2-\alpha}{2}} \|D^2\boldsymbol{u}\|^2\varphi^2 \, dx\\
&\qquad+(p-2-\alpha)\int_\Omega(\varepsilon +|D\boldsymbol{u}|^2)^{\frac{p-4-\alpha}{2}} \sum_{i=1}^n(D\boldsymbol{u}:D\boldsymbol{u}_{i})^2\varphi^2 \, dx\\
&\qquad-\alpha(p-2)\int_\Omega(\varepsilon +|D\boldsymbol{u}|^2)^{\frac{p-6-\alpha}{2}}\sum_{k=1}^N\langle D^2\boldsymbol u D\boldsymbol u, \nabla u^k\rangle \langle D^2\boldsymbol u D\boldsymbol u, \nabla u^k\rangle\varphi^2 \, dx.\\
&\quad\geq \int_\Omega(\varepsilon +|D\boldsymbol{u}|^2)^{\frac{p-2-\alpha}{2}} \|D^2\boldsymbol{u}\|^2\varphi^2 \, dx
\end{split}
\end{equation*}

Otherwise, if $p-2-\alpha-\alpha(p-2) < 0$, using \eqref{eq:FraCo}, it follows that:
\begin{equation*}
    \begin{split}
        &\int_\Omega(\varepsilon +|D\boldsymbol{u}|^2)^{\frac{p-2-\alpha}{2}} \|D^2\boldsymbol{u}\|^2\varphi^2 \, dx\\
&\qquad+(p-2-\alpha)\int_\Omega(\varepsilon +|D\boldsymbol{u}|^2)^{\frac{p-4-\alpha}{2}} \sum_{i=1}^n(D\boldsymbol{u}:D\boldsymbol{u}_{i})^2\varphi^2 \, dx\\
&\qquad-\alpha(p-2)\int_\Omega(\varepsilon +|D\boldsymbol{u}|^2)^{\frac{p-6-\alpha}{2}}\sum_{k=1}^N\langle D^2\boldsymbol u D\boldsymbol u, \nabla u^k\rangle \langle D^2\boldsymbol u D\boldsymbol u, \nabla u^k\rangle\varphi^2 \, dx.\\
&\quad\geq (1-\alpha)(p-1)\int_\Omega(\varepsilon +|D\boldsymbol{u}|^2)^{\frac{p-2-\alpha}{2}} \|D^2\boldsymbol{u}\|^2\varphi^2 \, dx
\end{split}
\end{equation*}

Therefore, summing up, we obtain
\begin{eqnarray}\label{FraCo11}
&&\int_\Omega(\varepsilon +|D\boldsymbol{u}|^2)^{\frac{p-2-\alpha}{2}} \|D^2\boldsymbol{u}\|^2\varphi^2 \, dx\\\nonumber
&&\quad+(p-2-\alpha)\int_\Omega(\varepsilon +|D\boldsymbol{u}|^2)^{\frac{p-4-\alpha}{2}} \sum_{i=1}^n(D\boldsymbol{u}:D\boldsymbol{u}_{i})^2\varphi^2 \, dx\\\nonumber&&\quad-\alpha(p-2)\int_\Omega(\varepsilon +|D\boldsymbol{u}|^2)^{\frac{p-6-\alpha}{2}}\sum_{k=1}^N\langle D^2\boldsymbol u D\boldsymbol u, \nabla u^k\rangle \langle D^2\boldsymbol u D\boldsymbol u, \nabla u^k\rangle\varphi^2 \, dx \\\nonumber
&&\qquad \geq \min\{1,(p-1)(1-\alpha)\}\int_\Omega(\varepsilon +|D\boldsymbol{u}|^2)^{\frac{p-2-\alpha}{2}} \|D^2\boldsymbol{u}\|^2\varphi^2 \, dx.
\end{eqnarray}

Now we estimate the right-hand side of \eqref{a_4}. To proceed, we observe that 
\begin{eqnarray}\label{aa}
\nonumber &&|(D\boldsymbol{u}_{i}: \boldsymbol{u}_{i}(\nabla \varphi)^T)|\leq |D\boldsymbol{u}_{i}| |\boldsymbol{u}_{i}(\nabla \varphi)^T|=|D\boldsymbol{u}_{i}|\sqrt{\sum_{j=1}^N\sum_{k=1}^n(u^j_{i}\varphi_k)^2}\\&&\qquad=|D\boldsymbol{u}_{i}|\sqrt{|\nabla \varphi|^2\sum_{j=1}^N(u^j_{i})^2}\leq |D\boldsymbol{u}_{i}| |D\boldsymbol{u}||\nabla \varphi|.
\end{eqnarray}
Using weighted Young inequality and \eqref{aa}, we get
\begin{eqnarray}\label{eq:13est1}\nonumber&&\mathcal{I}_1=-2\int_\Omega\frac{(\varepsilon +|D\boldsymbol{u}|^2)^{\frac{p-2}{2}}}{(\varepsilon +|D\boldsymbol{u}|^{2})^{\frac{\alpha}{2}}}(D\boldsymbol{u}_{i}: \boldsymbol{u}_{i}(\nabla \varphi)^T) \varphi \  dx\\
&&\quad\leq2\int_\Omega(\varepsilon +|D\boldsymbol{u}|^2)^{\frac{p-2-\alpha}{2}}|(D\boldsymbol{u}_{i}: \boldsymbol{u}_{i}(\nabla \varphi)^T)|\varphi \, dx\\\nonumber
&&\quad \leq 2\int_\Omega(\varepsilon +|D\boldsymbol{u}|^2)^{\frac{p-2-\alpha}{2}}|D\boldsymbol{u}_{i}||D\boldsymbol{u}||\nabla \varphi|\varphi \, dx\\\nonumber
&&\quad \leq 2\delta\int_\Omega(\varepsilon +|D\boldsymbol{u}|^2)^{\frac{p-2-\alpha}{2}}|D\boldsymbol{u}_{i}|^2 {\varphi^2} \, dx\\\nonumber
&&\qquad+\frac{1}{2\delta}\int_\Omega(\varepsilon +|D\boldsymbol{u}|^2)^{\frac{p-2-\alpha}{2}}|D\boldsymbol{u}|^2|\nabla \varphi|^2 \, dx\\
&&\quad \leq 2\delta\int_\Omega(\varepsilon +|D\boldsymbol{u}|^2)^{\frac{p-2-\alpha}{2}}|D\boldsymbol{u}_{i}|^2 {\varphi ^2}  \, dx\\\nonumber
&&\qquad+\frac{1}{2\delta}\int_\Omega{(\varepsilon +|D\boldsymbol{u}|^2)^{\frac{p-\alpha}{2}}}|\nabla \varphi|^2 \, dx.
\end{eqnarray}

Summing up on $i=1,\dots,n$, leads to
\begin{eqnarray}\label{magg1}
 &&\sum_{i=1}^n \mathcal{I}_1 \leq 2\sum_{i=1}^{n}\int_\Omega(\varepsilon +|D\boldsymbol{u}_{\varepsilon}|^2)^{\frac{p-2-\alpha}{2}}|(D\boldsymbol{u}_{i}: \boldsymbol{u}(\nabla \varphi)^T)|\varphi \, dx
 \\\nonumber
 &&\qquad \leq 2\delta\int_\Omega(\varepsilon +|D\boldsymbol{u}|^2)^{\frac{p-2-\alpha}{2}}\|D^2\boldsymbol{u}\|^2 {\varphi^2} \, dx\\\nonumber&&\quad\qquad+\frac{n}{2\delta}\int_\Omega{(\varepsilon +|D\boldsymbol{u}|^2)^{\frac{p-\alpha}{2}}}|\nabla \varphi|^2 \, dx.
\end{eqnarray}

As in \eqref{eq:13est1}, by weighted Young inequality, follows that
\begin{eqnarray}\label{magg2}
	&&\sum_{i=1}^n \mathcal{I}_2=\nonumber -2(p-2)\sum_{i=1}^n\int_\Omega\frac{(\varepsilon +|D\boldsymbol{u}|^2)^{\frac{p-4}{2}}}{(\varepsilon +|D\boldsymbol{u}|^{2})^{\frac{\alpha}{2}}}(D\boldsymbol{u}:D\boldsymbol{u}_{i})(D\boldsymbol{u}: \boldsymbol{u}_{i}(\nabla \varphi)^T)\varphi \, dx\\
    && \quad\qquad\leq\nonumber2|p-2|\sum_{i=1}^{n}\int_\Omega(\varepsilon +|D\boldsymbol{u}|^2)^{\frac{p-4-\alpha}{2}}|(D\boldsymbol{u}:D\boldsymbol{u}_{i})||(D\boldsymbol{u}: \boldsymbol{u}_{i}(\nabla \varphi)^T)|\varphi \, dx
	\\
	&&\quad \qquad \leq 2\delta|p-2|\int_\Omega(\varepsilon +|D\boldsymbol{u}|^2)^{\frac{p-2-\alpha}{2}}\|D^2\boldsymbol{u}\|^2{\varphi^2} \, dx\\\nonumber&&\qquad \qquad+\frac{n|p-2|}{2\delta}\int_\Omega{(\varepsilon +|D\boldsymbol{u}|^2)^{\frac{p-\alpha}{2}}}|\nabla \varphi|^2 \, dx.
\end{eqnarray}

Let us focus now on the boundary terms appearing in \eqref{a_4}, namely $\mathcal{I}_3$, $\mathcal{I}_4$ and  $\mathcal{I}_6$. To this aim, we need to recall the definition of the tangential gradient with respect to a regular level set. For a scalar function $w\in C^1(\overline \Omega)$, we set
\begin{equation*}
    \nabla_{T}w:=\nabla w-\langle \nabla w,\boldsymbol{\nu}\rangle \boldsymbol{\nu},
\end{equation*}
where $\boldsymbol{\nu}$ is the outward unit vector on $\partial \Omega$.
\begin{comment}
For the sake of clarity, let us set:
\begin{equation*}
    T_1=\sum_{i=1}^n \int_{\partial \Omega}(\varepsilon+|D{\boldsymbol u}|^2)^{\frac{p-2-\alpha}{2}}\langle (D\boldsymbol{u}_i)^T\boldsymbol{u}_i,\boldsymbol{\nu}\rangle\varphi^2\,d\mathcal{H}^{n-1}
\end{equation*}

\begin{equation*}
    T_2=(p-2)\sum_{i=1}^n\int_{\partial \Omega}(\varepsilon+|D{\boldsymbol u}|^2)^{\frac{p-4-\alpha}{2}}(D\boldsymbol u_i:D\boldsymbol u)\langle (D\boldsymbol{u})^T\boldsymbol{u}_i,\boldsymbol{\nu}\rangle\varphi^2\,d\mathcal{H}^{n-1}
\end{equation*}

\begin{equation*}
    T_3=\sum_{i=1}^n\int_{\partial \Omega}\langle \operatorname{\bf div}( (\varepsilon+|D{\boldsymbol u}|^2)^{\frac{p-2}{2}}D{\boldsymbol u}),\boldsymbol{u}_i\rangle \nu_i \frac{\varphi ^2}{(\varepsilon+|D{\boldsymbol u}|^2)^{\frac{\alpha}{2}}}\,d\mathcal{H}^{n-1}
\end{equation*}
\end{comment}

For simplicity let us treat separately the Dirichlet, the Neumann and the convex case.

\textbf{Dirichlet case.} Since $\boldsymbol{u}$ satisfies the boundary condition \eqref{Dir_cond}, then $\nabla_{T} u^{\beta} = 0$ on $\partial{\Omega}$ for $\beta=1,\dots,N$. As a consequence:
\begin{equation}
\boldsymbol{\nu} = \frac{\nabla u^{\beta}}{|\nabla u^{\beta}|},\quad \text{if } \nabla u^\beta \neq 0.
\end{equation}
for $\beta=1,\dots,N$. In particular, in the following computations, we will use the fact that:
\begin{equation}\label{normalD}
    \boldsymbol{\nu}=\left(\frac{u_1^{\alpha_1}}{|\nabla u^{\alpha_1}|},...,\frac{u_n^{\alpha_N}}{|\nabla u^{\alpha_N}|}\right), \quad \text{if} \quad |D\boldsymbol{u}|\neq 0,
\end{equation}
with $\alpha_1,...,\alpha_N\in \{1,...,N\}$.
Otherwise, if $\nabla u^\beta =0$, the corresponding boundary term is identically zero.

So that, by \eqref{normalD}, we can rewrite boundary term, $\mathcal{I}_6$, as follows

\begin{equation}\label{boundary}
    \begin{split}
-\sum_{i=1}^n \mathcal{I}_6=\sum_{i=1}^n& \int_{\partial \Omega}\langle \operatorname{\bf div}( (\varepsilon+|D{\boldsymbol u}|^2)^{\frac{p-2}{2}}D{\boldsymbol u}),\boldsymbol{u}_i\rangle \nu_i \frac{\varphi ^2}{(\varepsilon+|D{\boldsymbol u}|^2)^{\frac{\alpha}{2}}}\,d\mathcal{H}^{n-1}\\
        &=\int_{\partial \Omega} \sum_{\beta=1}^{N} \operatorname{div}((\varepsilon+|D{\boldsymbol u}|^2)^{\frac{p-2}{2}} \nabla u^\beta) \sum_{i=1}^{n} u_i^\beta \nu_i \frac{\varphi ^2}{(\varepsilon+|D{\boldsymbol u}|^2)^{\frac{\alpha}{2}}} \,d\mathcal{H}^{n-1}\\
&= (p-2)\sum_{i=1}^n\int_{\partial \Omega}(\varepsilon+|D{\boldsymbol u}|^2)^{\frac{p-4-\alpha}{2}}\sum_{\beta=1}^N \langle D^2 \boldsymbol{u}D\boldsymbol{u},\nabla u^\beta\rangle  u_i^\beta\nu_i \varphi^2\,d\mathcal{H}^{n-1}\\
&\quad+\int_{\partial \Omega} (\varepsilon+|D{\boldsymbol u}|^2)^{\frac{p-2-\alpha}{2}} \sum_{\beta=1}^{N} \Delta u^\beta |\nabla u^\beta| \varphi^2 \,d\mathcal{H}^{n-1},\\
&= (p-2)\sum_{i=1}^n\int_{\partial \Omega}(\varepsilon+|D{\boldsymbol u}|^2)^{\frac{p-4-\alpha}{2}}\langle (D\boldsymbol{u_i})^T\boldsymbol{u}_i,\boldsymbol{\nu}\rangle |D{\boldsymbol u}|^2 \varphi^2\,d\mathcal{H}^{n-1}\\
&\quad+\int_{\partial \Omega} (\varepsilon+|D{\boldsymbol u}|^2)^{\frac{p-2-\alpha}{2}} \sum_{\beta=1}^{N} \Delta u^\beta |\nabla u^\beta| \varphi^2 \,d\mathcal{H}^{n-1},
    \end{split}
\end{equation}
where, we used the identity:
\begin{equation*}
    \sum_{\beta=1}^N \langle D^2 \boldsymbol{u}D\boldsymbol{u},\nabla u^\beta\rangle \sum_{i=1}^n u_i^\beta\nu_i = \sum_{i=1}^n \langle (D\boldsymbol{u}_i)^T\boldsymbol{u}_i,\nu \rangle |D\boldsymbol{u}|^2.
\end{equation*}

Now, by \cite[Equation $(3,1,1,8)$]{Gr} (see also \cite[Equation (3.44)]{Cma}), we recall that 
\begin{equation}\label{Nonsmooth_dom}
\begin{split}
    \Delta u^\beta \frac{\partial u^\beta}{\partial \boldsymbol{\nu}} - &\sum_{i,j=1}^n u_{ij}^\beta u_i^\beta \nu_j\\
    &= \operatorname{div}_T \left ( \frac{\partial u^\beta}{\partial \boldsymbol \nu} \nabla_T u^\beta\right ) - \operatorname{tr}\mathcal B \left (\frac{\partial u^\beta}{\partial \boldsymbol \nu} \right)^2 \\
    &\qquad- \mathcal{B}(\nabla_Tu^\beta,\nabla_Tu^\beta) - 2 \langle \nabla_Tu^\beta, \nabla_T\frac{\partial u^\beta}{\partial \boldsymbol \nu}\rangle \quad \text{on } \partial \Omega,
\end{split}
\end{equation}
for $\beta=1,...,N$, where $\operatorname{tr} \mathcal{B}$ is the trace of $\mathcal{B}$, $\operatorname{div}_T$ and $\nabla_T$ denote the divergence and the gradient operator on $\partial\Omega$.

By \eqref{Nonsmooth_dom}, since $\nabla_Tu^\beta = 0$ and by \eqref{normalD}, the boundary term, $\mathcal{I}_4$, becomes
\begin{equation}\label{pezzobordo1}
\begin{split}
    \sum_{i=1}^n \mathcal{I}_4&=\sum_{i=1}^{n} \int_{\partial \Omega}(\varepsilon+|D{\boldsymbol u}|^2)^{\frac{p-2-\alpha}{2}}\langle (D\boldsymbol{u}_i)^T\boldsymbol{u}_i,\boldsymbol{\nu}\rangle\varphi^2\,d\mathcal{H}^{n-1}\\
    &\quad= \int_{\partial \Omega}(\varepsilon+|D{\boldsymbol u}|^2)^{\frac{p-2-\alpha}{2}}\sum_{\beta=1}^{N} \Delta u^\beta |\nabla u^\beta| \varphi^2\,d\mathcal{H}^{n-1}\\
    &\qquad+ \int_{\partial \Omega}(\varepsilon+|D{\boldsymbol u}|^2)^{\frac{p-2-\alpha}{2}} \operatorname{tr}\mathcal{B} \sum_{\beta=1}^{N} |\nabla u^\beta|^2 \varphi^2\,d\mathcal{H}^{n-1}.
\end{split}
\end{equation}

Now we rewrite the boundary term $\mathcal{I}_3$.
Using \eqref{normalD}, it is easy to check that:
\begin{equation*}
    \langle (D \boldsymbol{u})^T\boldsymbol{u}_i, \boldsymbol{\nu}\rangle = \nu_i \sum_{\beta=1}^N{|\nabla u^\beta|}^2,
\end{equation*}

and
\begin{equation*}
    \langle (D\boldsymbol{u}_i)^T\boldsymbol{u}_i, \boldsymbol{\nu}\rangle= \sum_{\beta=1}^N \sum_{j=1}^n u_{ji}^\beta u_j^\beta \nu_i.
\end{equation*}
Therefore, the following identity holds
\begin{equation}\label{boundaryidentity}
|D{\boldsymbol u}|^2 \langle (D\boldsymbol{u}_i)^T\boldsymbol{u}_i,\boldsymbol{\nu}\rangle = (D\boldsymbol u_i:D\boldsymbol u)\langle (D\boldsymbol{u})^T\boldsymbol{u}_i,\boldsymbol{\nu}\rangle,
\end{equation}
for $i=1,...,n$.

By \eqref{boundaryidentity}, we deduce
\begin{equation}\label{stimaI_3}
\begin{split}
    \mathcal{I}_3 &=(p-2)\int_{\partial \Omega}(\varepsilon+|D{\boldsymbol u}|^2)^{\frac{p-4-\alpha}{2}}(D\boldsymbol u_i:D\boldsymbol u)\langle (D\boldsymbol{u})^T\boldsymbol{u}_i,\boldsymbol{\nu}\rangle\varphi^2\,d\mathcal{H}^{n-1} \\
    &=(p-2)\int_{\partial \Omega}(\varepsilon+|D{\boldsymbol u}|^2)^{\frac{p-4-\alpha}{2}}|D{\boldsymbol u}|^2 \langle (D\boldsymbol{u}_i)^T\boldsymbol{u}_i,\boldsymbol{\nu}\rangle \varphi^2\,d\mathcal{H}^{n-1}
\end{split}
\end{equation}

Summing up, using \eqref{boundary}, \eqref{pezzobordo1} and \eqref{stimaI_3}, we obtain the resulting boundary term
\begin{equation}\label{trB}
\begin{split}
    \sum_{i=1}^n \left(\mathcal{I}_3+\mathcal{I}_4+\mathcal{I}_6\right )&=\int_{\partial \Omega}(\varepsilon+|D{\boldsymbol u}|^2)^{\frac{p-2-\alpha}{2}} \operatorname{tr}\mathcal{B} |D{\boldsymbol u}|^2 \varphi^2\,d\mathcal{H}^{n-1}\\
    &\leq C(n,N) \int_{\partial \Omega} (\varepsilon+|D{\boldsymbol u}|^2)^{\frac{p-2-\alpha}{2}}|D\boldsymbol{u}|^2|\mathcal{B}|\varphi ^2 \,d\mathcal{H}^{n-1},
\end{split}
\end{equation}
where $C(n,N)$ is a positive constant. Now we assume that $\varphi \in C^{\infty}_c(B_r(x_0))$, for some $x_0\in \partial \Omega$ and $r\in (0,R_{\Omega})$, where $R_\Omega$ is given by Lemma \ref{antoninho}.
By \eqref{antony}, with $\rho =|\mathcal{B}|$ and $v=\varphi(\varepsilon+|D\boldsymbol{u}|^2)^{\frac{p-2-\alpha}{4}}u_i^{\beta}$, for $i=1,...,n$ and $\beta=1,...,N$, we deduce that 

\begin{eqnarray}\label{penultimotermine}
&&\sum_{i=1}^n \left(\mathcal{I}_3+\mathcal{I}_4+\mathcal{I}_6\right )=\int_{\partial \Omega} (\varepsilon+|D{\boldsymbol u}|^2)^{\frac{p-2-\alpha}{2}}|D\boldsymbol{u}|^2|\mathcal{B}|\varphi ^2 \,d\mathcal{H}^{n-1} \\\nonumber &&\qquad \leq C(n,N,p,\alpha, L_\Omega)K_{\Omega}(r)\left(\int_{\Omega\cap B_r(x_0)}(\varepsilon+|D{\boldsymbol u}|^2)^{\frac{p-2-\alpha}{2}}\|D^2\boldsymbol{u}\|^2\varphi ^2\,dx\right.\\\nonumber &&\qquad\qquad\qquad\qquad\qquad\qquad \left.  + \int_{\Omega\cap B_r(x_0)}(\varepsilon+|D{\boldsymbol u}|^2)^{\frac{p-2-\alpha}{2}}|D\boldsymbol{u}|^2|\nabla\varphi| ^2 \,dx \right),
  \end{eqnarray} 
where $C(n,N,p,\alpha, L_\Omega)$ is a positive constant.

Using the estimates of the left hand side of \eqref{a_4}, (see \eqref{K_N} and \eqref{FraCo11}), and the ones of the right hand side of \eqref{a_4}, (see  \eqref{magg1}, \eqref{magg2},  and \eqref{penultimotermine}), it follows that:

if $1<p\leq 2$
\begin{equation*}
    \begin{split}
        &  K_N(\alpha,p)\int_\Omega(\varepsilon +|D\boldsymbol{u}|^2)^{\frac{p-2-\alpha}{2}} \|D^2\boldsymbol{u}\|^2\varphi^2 \, dx \\  &\leq 2\delta(|p-2|+1)\int_\Omega(\varepsilon +|D\boldsymbol{u}|^2)^{\frac{p-2-\alpha}{2}}\|D^2\boldsymbol{u}\|^2 {\varphi^2} \, dx\\&\quad+\frac{n}{2\delta}(|p-2|+1)\int_\Omega{(\varepsilon +|D\boldsymbol{u}|^2)^{\frac{p-\alpha}{2}}}|\nabla \varphi|^2 \, dx \\ &\quad + C(n,N,p,\alpha, L_\Omega)K_{\Omega}(r)\left(\int_{\Omega\cap B_r(x_0)}(\varepsilon+|D{\boldsymbol u}|^2)^{\frac{p-2-\alpha}{2}}\|D^2\boldsymbol{u}\|^2\varphi ^2\,dx\right.\\& \qquad\qquad\qquad\qquad\qquad\qquad\left.  + \int_{\Omega\cap B_r(x_0)}(\varepsilon+|D{\boldsymbol u}|^2)^{\frac{p-2-\alpha}{2}}|D\boldsymbol{u}|^2|\nabla\varphi ^2| \,dx \right) \\
        &\quad+ \sum_{i=1}^n\int_\Omega \langle \operatorname{\bf div}( (\varepsilon+|D{\boldsymbol u}|^2)^{\frac{p-2}{2}}D{\boldsymbol u}),\partial _{x_i} \left(\frac{ \boldsymbol{u}_i \varphi ^2}{(\varepsilon+|D{\boldsymbol u}|^2)^{\frac{\alpha}{2}}} \right) \rangle\,dx;
    \end{split}
\end{equation*}

if $p>2$

\begin{equation*}
    \begin{split}
        &\min \{1,(p-1)(1-\alpha)\}\int_\Omega(\varepsilon +|D\boldsymbol{u}|^2)^{\frac{p-2-\alpha}{2}} \|D^2\boldsymbol{u}\|^2\varphi^2 \, dx \\ & \leq 2\delta(|p-2|+1)\int_\Omega(\varepsilon +|D\boldsymbol{u}|^2)^{\frac{p-2-\alpha}{2}}\|D^2\boldsymbol{u}\|^2 {\varphi^2} \, dx\\&\quad+\frac{n}{2\delta}(|p-2|+1)\int_\Omega{(\varepsilon +|D\boldsymbol{u}|^2)^{\frac{p-\alpha}{2}}}|\nabla \varphi|^2 \, dx \\ &\quad + C(n,N,p,\alpha, L_\Omega)K_{\Omega}(r)\left(\int_{\Omega\cap B_r(x_0)}(\varepsilon+|D{\boldsymbol u}|^2)^{\frac{p-2-\alpha}{2}}\|D^2\boldsymbol{u}\|^2\varphi ^2\,dx\right.\\& \qquad\qquad\qquad\qquad\qquad\qquad\left.  + \int_{\Omega\cap B_r(x_0)}(\varepsilon+|D{\boldsymbol u}|^2)^{\frac{p-2-\alpha}{2}}|D\boldsymbol{u}|^2|\nabla\varphi |^2 \,dx \right) \\
        & \quad+\sum_{i=1}^n\int_\Omega \langle \operatorname{\bf div}( (\varepsilon+|D{\boldsymbol u}|^2)^{\frac{p-2}{2}}D{\boldsymbol u}),\partial _{x_i} \left(\frac{ \boldsymbol{u}_i \varphi ^2}{(\varepsilon+|D{\boldsymbol u}|^2)^{\frac{\alpha}{2}}} \right) \rangle\,dx.
    \end{split}
\end{equation*}

There exists a positive constant $C(p,\alpha)$, for any $p >1$, such that 

\begin{equation}\label{zappa}
    \begin{split}
        &(C(p,\alpha)-2\delta(|p-2|+1)\\&\qquad\qquad-C(n,N,p,\alpha, L_\Omega)K_{\Omega}(r))\int_\Omega(\varepsilon +|D\boldsymbol{u}|^2)^{\frac{p-2-\alpha}{2}} \|D^2\boldsymbol{u}\|^2\varphi^2 \, dx \\ & \leq \left(\frac{n}{2\delta}(|p-2|+1)+C(n,N,p,\alpha, L_\Omega)K_{\Omega}(r)\right)\int_\Omega{(\varepsilon +|D\boldsymbol{u}|^2)^{\frac{p-\alpha}{2}}}|\nabla \varphi|^2 \, dx  \\
        & \qquad+\sum_{i=1}^n\int_\Omega \langle \operatorname{\bf div}( (\varepsilon+|D{\boldsymbol u}|^2)^{\frac{p-2}{2}}D{\boldsymbol u}),\partial _{x_i} \left(\frac{ \boldsymbol{u}_i}{(\varepsilon+|D{\boldsymbol u}|^2)^{\frac{\alpha}{2}}} \right) \rangle \varphi ^2\,dx\\
        &\qquad+2\sum_{i=1}^n\int_\Omega \langle \operatorname{\bf div}( (\varepsilon+|D{\boldsymbol u}|^2)^{\frac{p-2}{2}}D{\boldsymbol u}),\boldsymbol{u}_i \rangle \frac{\varphi \varphi_i}{(\varepsilon+|D{\boldsymbol u}|^2)^{\frac{\alpha}{2}}} \,dx
    \end{split}
\end{equation}
Moreover, by standard Young inequality, we have
\begin{equation}\label{phi_term}
    \begin{split}
    2\sum_{i=1}^n&\int_\Omega \langle \operatorname{\bf div}( (\varepsilon+|D{\boldsymbol u}|^2)^{\frac{p-2}{2}}D{\boldsymbol u}),\boldsymbol{u}_i \rangle \frac{\varphi \varphi_i}{(\varepsilon+|D{\boldsymbol u}|^2)^{\frac{\alpha}{2}}} \,dx\\
    &\leq C(n) \left(\int_\Omega |\operatorname{\bf div}( (\varepsilon+|D{\boldsymbol u}|^2)^{\frac{p-2}{2}}D{\boldsymbol u})| |D{\boldsymbol u}|^{1-\alpha} \varphi^2 \,dx \right.\\
    & \quad \qquad\left. +\int_\Omega |\operatorname{\bf div}( (\varepsilon+|D{\boldsymbol u}|^2)^{\frac{p-2}{2}}D{\boldsymbol u})| |D{\boldsymbol u}|^{1-\alpha} |\nabla \varphi|^2 \,dx \right).
    \end{split}
\end{equation}

If condition \eqref{condizionesuk} is fulfilled with $\overline c=C(p,\alpha)/C(n,N,p,\alpha,L_\Omega)$, for $\delta>0$ sufficiently small, there exists a constant $C>0$, depending on $\Omega$ only through $L_\Omega$, $d_\Omega$ and $r'\in (0,R_\Omega)$ depending also on $K$, such that:

\begin{equation*}
\begin{split}
    C(p,\alpha)&-2\delta(|p-2|+1)-C(n,N,p,\alpha, L_\Omega)K_{\Omega}(r) \\
    &\geq C(p,\alpha)-2\delta(|p-2|+1)-C(n,N,p,\alpha, L_\Omega)K(r)\geq C,
\end{split}
\end{equation*}
with $r \in (0,r']$.

Therefore, by \eqref{phi_term}, the estimate \eqref{zappa} becomes

\begin{equation}\label{abc}
    \begin{split}
        &\int_\Omega(\varepsilon +|D\boldsymbol{u}|^2)^{\frac{p-2-\alpha}{2}} \|D^2\boldsymbol{u}\|^2\varphi^2 \, dx \\ 
        & \leq \mathcal{C}\left(\int_\Omega{(\varepsilon +|D\boldsymbol{u}|^2)^{\frac{p-\alpha}{2}}}|\nabla \varphi|^2 \, dx\right.  \\
        &  \qquad\left.+\sum_{i=1}^n\int_\Omega \langle \operatorname{\bf div}( (\varepsilon+|D{\boldsymbol u}|^2)^{\frac{p-2}{2}}D{\boldsymbol u}),\partial _{x_i} \left(\frac{ \boldsymbol{u}_i}{(\varepsilon+|D{\boldsymbol u}|^2)^{\frac{\alpha}{2}}} \right) \rangle \varphi ^2\,dx\right. \\
        &\qquad\left. +\int_\Omega |\operatorname{\bf div}( (\varepsilon+|D{\boldsymbol u}|^2)^{\frac{p-2}{2}}D{\boldsymbol u})| |D{\boldsymbol u}|^{1-\alpha} \varphi^2 \,dx \right.\\
    & \qquad\left. +\int_\Omega |\operatorname{\bf div}( (\varepsilon+|D{\boldsymbol u}|^2)^{\frac{p-2}{2}}D{\boldsymbol u})| |D{\boldsymbol u}|^{1-\alpha} |\nabla \varphi|^2 \,dx \right),
    \end{split}
\end{equation}
for some positive constant $\mathcal{C}=\mathcal{C}(n,N,p,\alpha,L_\Omega)$, if $r \in (0,r']$.

In the case $x_0\in \Omega$ and $B_r(x_0)\subset\Omega$ inequality \eqref{abc} holds. Indeed the boundary term
in \eqref{a_4} vanishes. Moreover the constant $\mathcal{C}$ is independent of $L_\Omega$ and $K_{\Omega}$.

Now there exists $r'' \in (0,r')$, hence depending on $L_\Omega$, $d_\Omega$ and $K$ such that $\overline \Omega$ admits a finite covering $\{B_{r_k}\}$, with $r''\leq r_k \leq r'$, and a family of functions $\{\varphi_{r_k}\}$
such that $\varphi_{r_k} \in C_c^{\infty}(B_{r_k})$ and $\{\varphi^2_{r_k}\}$ is a partition of unity of $\overline \Omega$ associated with the covering $\{B_{r_k}\}$. Thus, $\sum_k \varphi^2_{r_k}=1$ in $\overline \Omega$. Moreover the functions $\varphi_{r_k}$ can also chosen so that $|\nabla \varphi_{r_k}|\leq C/r_k\leq C / r''$. Using \eqref{abc} with $\varphi=\varphi_{r_k}$, we obtain 

\begin{equation}\label{abcd}
    \begin{split}
        &\int_\Omega(\varepsilon +|D\boldsymbol{u}|^2)^{\frac{p-2-\alpha}{2}} \|D^2\boldsymbol{u}\|^2 \, dx \\ & \leq \mathcal{C}\left(\int_\Omega{(\varepsilon +|D\boldsymbol{u}|^2)^{\frac{p-\alpha}{2}}} \, dx\right.  \\
        &  \qquad\left.+\sum_{i=1}^n\int_\Omega \langle \operatorname{\bf div}( (\varepsilon+|D{\boldsymbol u}|^2)^{\frac{p-2}{2}}D{\boldsymbol u}),\partial _{x_i} \left(\frac{ \boldsymbol{u}_i}{(\varepsilon+|D{\boldsymbol u}|^2)^{\frac{\alpha}{2}}} \right) \rangle\,dx\right. \\
        &\qquad\left. +\int_\Omega |\operatorname{\bf div}( (\varepsilon+|D{\boldsymbol u}|^2)^{\frac{p-2}{2}}D{\boldsymbol u})| |D{\boldsymbol u}|^{1-\alpha} \,dx \right),
    \end{split}
\end{equation}

where $\mathcal{C}=\mathcal{C}(n,N,p,\alpha,L_\Omega,d_\Omega,K)$ is a positive constant and $d_\Omega$ is the diameter of $\Omega$. This concludes the proof of our theorem in the case condition \eqref{Dir_cond} is in force.

\vspace{0.5 cm}

\textbf{Neumann case.} Regarding the case with homogeneous Neumann on the boundary, the idea is the same. In the following we outline the main changes in the proof.\\
Starting from \eqref{a_4},we notice that, since $\boldsymbol u$ obeys \eqref{Neu_cond}, $\mathcal{I}_3=0$ for any $i=1,...,n$ and $\sum_{i=1}^n \mathcal{I}_6=0$.

Now we estimate the term $I_4$ in \eqref{a_4}. Recalling \eqref{Nonsmooth_dom}, since $\frac{\partial u^\beta}{\partial \boldsymbol{\nu}}=0$, we deduce
\begin{equation}\label{nonsmoothdom}
    \sum_{i=1}^n \langle (D\boldsymbol{u}_i)^T\boldsymbol{u}_i,\boldsymbol{\nu}\rangle = \sum_{\beta=1}^N \langle \nabla^2 u^\beta \nabla u^\beta,\boldsymbol{\nu}\rangle = 
    \sum_{\beta=1}^N \mathcal{B}(\nabla_{T} u^\beta,\nabla_{T} u^\beta).
\end{equation}
Therefore, by \eqref{nonsmoothdom},
\begin{equation}\label{boundNeum2}
\begin{split}
    \mathcal{I}_4=\sum_{i=1}^n&\int_{\partial \Omega}(\varepsilon+|D{\boldsymbol u}|^2)^{\frac{p-2-\alpha}{2}}\langle (D\boldsymbol{u}_i)^T\boldsymbol{u}_i,\boldsymbol{\nu}\rangle\varphi^2\,d\mathcal{H}^{n-1}\\
    &=\sum_{\beta=1}^N \int_{\partial \Omega}(\varepsilon+|D{\boldsymbol u}|^2)^{\frac{p-2-\alpha}{2}} \mathcal{B}(\nabla_{T} u^\beta,\nabla_{T} u^\beta)\varphi^2\,d\mathcal{H}^{n-1}\\
    &\leq C(n,N) \int_{\partial \Omega} (\varepsilon+|D{\boldsymbol u}|^2)^{\frac{p-2-\alpha}{2}}|D\boldsymbol{u}|^2|\mathcal{B}|\varphi ^2 \,d\mathcal{H}^{n-1} 
\end{split}
\end{equation}
With these variants, the analogous procedure as in the Dirichlet case yields (see\eqref{zappa})
\begin{equation}\label{abcdNeu}
    \begin{split}
        &\int_\Omega(\varepsilon +|D\boldsymbol{u}|^2)^{\frac{p-2-\alpha}{2}} \|D^2\boldsymbol{u}\|^2 \varphi^2\, dx \\ & \quad\leq \mathcal{C}\left(\int_\Omega{(\varepsilon +|D\boldsymbol{u}|^2)^{\frac{p-\alpha}{2}}} \, \varphi^2dx\right.  \\
        &  \quad\qquad\left.+\int_\Omega \langle \operatorname{\bf div}( (\varepsilon+|D{\boldsymbol u}|^2)^{\frac{p-2}{2}}D{\boldsymbol u}),\partial _{x_i} \left(\frac{ \boldsymbol{u}_i \varphi ^2}{(\varepsilon+|D{\boldsymbol u}|^2)^{\frac{\alpha}{2}}} \right) \rangle\,dx\right),
    \end{split}
\end{equation}
where $\mathcal{C}=C(n,N,p,\alpha,L_\Omega,d_\Omega,K)$ is a positive constant.

Integrating by parts in $x_i$ the last term of \eqref{abcdNeu} and using \eqref{Neu_cond}, we obtain:
\begin{equation*}%\label{abcdNeu}
    \begin{split}
        &\int_\Omega(\varepsilon +|D\boldsymbol{u}|^2)^{\frac{p-2-\alpha}{2}} \|D^2\boldsymbol{u}\|^2 \varphi^2\, dx \\ & \quad\leq \mathcal{C}\left(\int_\Omega{(\varepsilon +|D\boldsymbol{u}|^2)^{\frac{p-\alpha}{2}}} \, \varphi^2dx\right.  \\
        &  \quad\qquad\left.-\int_\Omega \langle \partial_{x_i}(\operatorname{\bf div}( (\varepsilon+|D{\boldsymbol u}|^2)^{\frac{p-2}{2}}D{\boldsymbol u})),\boldsymbol{u}_i\rangle\frac{ \varphi ^2}{(\varepsilon+|D{\boldsymbol u}|^2)^{\frac{\alpha}{2}}}\rangle\,dx\right),
    \end{split}
\end{equation*}
where $\mathcal{C}=C(n,N,p,\alpha,L_\Omega,d_\Omega,K)$ is a positive constant.

Finally, by the partition of unity and reintegrating by parts in the variable $x_i$ we obtain the thesis.

\vspace{0.5 cm}

\textbf{Convex case}. In particular, if $\Omega$ is a convex set, in the Dirichlet case, we can neglect the boundary term \eqref{trB} since it is nonpositive, indeed $$\operatorname{tr}\mathcal{B} \leq 0 \quad \text{on }\partial\Omega.
$$

Proceeding as in the Dirichlet case, we have 
\begin{equation}\label{abcdeh}
    \begin{split}
        &\int_\Omega(\varepsilon +|D\boldsymbol{u}|^2)^{\frac{p-2-\alpha}{2}} \|D^2\boldsymbol{u}\|^2 \, dx \\ & \leq \mathcal{C}\left(\int_\Omega{(\varepsilon +|D\boldsymbol{u}|^2)^{\frac{p-\alpha}{2}}} \, dx\right.  \\&  \qquad\left.+\sum_{i=1}^n\int_\Omega \langle \operatorname{\bf div}( (\varepsilon+|D{\boldsymbol u}|^2)^{\frac{p-2}{2}}D{\boldsymbol u}), \partial _{x_i}\left(\frac{\boldsymbol{u}_i}{(\varepsilon+|D{\boldsymbol u}|^2)^{\frac{\alpha}{2}}}\right)\rangle\,dx\right.\\
        &\qquad\left. +\int_\Omega |\operatorname{\bf div}( (\varepsilon+|D{\boldsymbol u}|^2)^{\frac{p-2}{2}}D{\boldsymbol u})| |D{\boldsymbol u}|^{1-\alpha} \,dx \right).
    \end{split}
\end{equation}
where $\mathcal{C}=C(n,N,p,\alpha,L_\Omega,d_\Omega)$ is a positive constant. We stress that in this case the quantity $K_\Omega$ does not play any role.

Regarding the Neumann case, if we assume that $\Omega$ is convex, the term \eqref{boundNeum2} is nonpositive since
\begin{equation*}
    \mathcal{B}(\nabla_{T} u^\beta,\nabla_{T} u^\beta) \leq 0 \quad \text{on } \partial{\Omega}.
\end{equation*}

With the analogous procedure as in the Neumann case, we deduce 
\begin{equation}\label{abcdefgh}
    \begin{split}
        &\int_\Omega(\varepsilon +|D\boldsymbol{u}|^2)^{\frac{p-2-\alpha}{2}} \|D^2\boldsymbol{u}\|^2 \, dx \\ & \leq \mathcal{C}\left(\int_\Omega{(\varepsilon +|D\boldsymbol{u}|^2)^{\frac{p-\alpha}{2}}} \, dx\right.  \\&  \qquad\left.+\sum_{i=1}^n\int_\Omega \langle \operatorname{\bf div}( (\varepsilon+|D{\boldsymbol u}|^2)^{\frac{p-2}{2}}D{\boldsymbol u}), \partial _{x_i}\left(\frac{\boldsymbol{u}_i}{(\varepsilon+|D{\boldsymbol u}|^2)^{\frac{\alpha}{2}}}\right)\rangle\,dx\right),
    \end{split}
\end{equation}

Thus, estimate \eqref{abcdefgh} still holds with a constant $\mathcal{C}$ independent on $K_{\Omega}$.

\end{proof}
%\newpage
To prove our main result, we need the following technical lemma (see \cite[Lemma 5.2]{Cma} or \cite{Ant}):
\vspace{0.5 cm}

\begin{lemma}\label{approxcap}
Let $\Omega$ be a bounded Lipschitz domain in $\R^n$, $n \geq 2$ such that $\partial \Omega \in W^{2,1}$. Assume that the function $\mathcal K_\Omega (r)$, defined as in \eqref{KB}, is finite-valued for $r\in (0,1)$.
Then there exist positive constants $r_0$ and $C$ and a sequence of bounded open sets $\{\Omega_m\}$, 
such that $\partial \Omega _m \in C^\infty$, $\Omega \subset \Omega _m$, $\lim _{m \to \infty}|\Omega _m \setminus \Omega| = 0$, the Hausdorff distance between $\Omega _m$ and $\Omega$ tends to $0$ as $m \to \infty$,
\begin{equation}\label{prop1}
L_{\Omega _m} \leq C L_\Omega \,, \quad d_{\Omega _m} \leq C d_\Omega
\end{equation}
and
\begin{equation}\label{prop2}
\mathcal K_{\Omega_m}(r) \leq C \mathcal K_{\Omega} (r)
\end{equation}
for $r\in (0, r_0)$ and $m \in \mathbb N$.
\end{lemma}

\vspace{0.3 cm}

\begin{proof}[Proof of Theorem \ref{teo1INTRO}.] The thesis formally follows from an application of Theorem \ref{stimaC^2}. However, in our setup, lack of regularity of the weak solution $\boldsymbol u$ and of the domain $\Omega$ preclude a direct application of the previous results and need a regularization argument involving the domain and the right hand side of the system.
%\textbf{Local estimates.} See \cite{MMSV}.
%\textbf{Global estimates.}
\vspace{0.1 cm}

\textbf{Dirichlet case.} First of all, we analyze the Dirichlet problem \eqref{system1} + \eqref{Dir_cond}. We split the proof into several steps.

\textit{Step $1.$} Here, we assume 
\begin{equation}\label{bordoregolare}
   \partial \Omega\in C^2,
\end{equation}
\vspace{0.05cm}
\begin{equation}\label{reg_f}
    f \in C^\infty(\Omega).
\end{equation}
Given $\varepsilon \in (0,1)$, we consider a weak solution $\boldsymbol{u_\varepsilon}$ to the approximating system
\beq\label{system33}\tag{$p$-$D_\varepsilon$}
\begin{cases}
-\operatorname{\bf div}((\varepsilon+|D{\boldsymbol u_\varepsilon}|^2)^\frac{p-2}{2}D{\boldsymbol u_\varepsilon})=   {\boldsymbol f}(x)& \mbox{in $\Omega$}\\
\boldsymbol u_\varepsilon = 0  & \mbox{on  $\partial \Omega$}.
\end{cases}
\eeq
By \cite[Theorem $2.1$]{Cma2}, we remark that 
\begin{equation}\label{linfinito}
    \|D\boldsymbol{u}_\varepsilon\|_{L^{\infty}(\Omega)}\leq C(p,n,|\Omega|,\|\operatorname{tr}\mathcal{B}\|_{L^{n-1,1}(\Omega)},\|\boldsymbol f\|_{L^q(\Omega)}),
\end{equation}
where $C(p,n,|\Omega|,\|\operatorname{tr}\mathcal{B}\|_{L^{n-1,1}(\Omega)},\|\boldsymbol f\|_{L^q(\Omega)})$ is a positive constant, with $q>n$. As shown in the proof of that theorem $\boldsymbol{u_\varepsilon}\in W^{1,2}_0(\Omega)\cap W^{1,\infty}(\Omega)\cap W^{2,2}(\Omega)$, and by standard approximation, see \cite[Chapter $2$, Theorem $1$]{Bur}, there exists a sequence $\{\boldsymbol u_k\}\subset C^{\infty}(\Omega)\cap C^2(\overline \Omega)$, satisfying $\boldsymbol{u}_k=0$ on $\partial \Omega$ for $k\in \mathbb{N}$, and: 
\begin{equation*}
\boldsymbol{u}_k\rightarrow\boldsymbol{u}_\varepsilon \quad \text{in } W_0^{1,2}(\Omega), \quad \boldsymbol{u}_k\rightarrow\boldsymbol{u}_\varepsilon \quad \text{in } W^{2,2}(\Omega), \quad D\boldsymbol u_k\rightarrow D\boldsymbol u_\varepsilon \quad \text{a.e. in }\Omega,
\end{equation*}
as $k \rightarrow \infty$. Moreover,
\begin{equation}\label{bound_unif_k}
\|D\boldsymbol{u}_k\|_{L^{\infty}(\Omega)}\leq C\|D\boldsymbol{u}_\varepsilon\|_{L^{\infty}(\Omega)},
\end{equation}
for some constants $C$ independent on $k$.
From Theorem \ref{stimaC^2}, using the estimate \eqref{abcdefDir}, with $\boldsymbol{u}=\boldsymbol{u}_k$, we get 
\begin{equation}\label{abcdefg}
    \begin{split}
        &\int_\Omega(\varepsilon +|D\boldsymbol{u}_k|^2)^{\frac{p-2-\alpha}{2}} \|D^2\boldsymbol{u}_k\|^2 \, dx \\ & \leq \mathcal{C}\left(\int_\Omega{(\varepsilon +|D\boldsymbol{u}_k|^2)^{\frac{p-\alpha}{2}}} \, dx\right.  \\&  \qquad\left.+\sum_{i=1}^n\int_\Omega \langle \operatorname{\bf div}( (\varepsilon+|D{\boldsymbol u}_k|^2)^{\frac{p-2}{2}}D{\boldsymbol u}_k), \partial _{x_i}\left(\frac{\boldsymbol{u}_{k,i}}{(\varepsilon+|D{\boldsymbol u}_k|^2)^{\frac{\alpha}{2}}}\right)\rangle\,dx\right.\\
        &\qquad\left. +\int_\Omega |\operatorname{\bf div}( (\varepsilon+|D{\boldsymbol u_k}|^2)^{\frac{p-2}{2}}D{\boldsymbol u_k})| |D{\boldsymbol u_k}|^{1-\alpha} \,dx \right),
    \end{split}
\end{equation}
where $\mathcal{C}(n,N,p,\alpha,L_\Omega,d_\Omega,K_{\Omega})$ is a positive constant.

By the dominated convergence theorem, recalling $-\operatorname{\bf div}((\varepsilon+|D{\boldsymbol u_\varepsilon}|^2)^\frac{p-2}{2}D{\boldsymbol u_\varepsilon})=   {\boldsymbol f}(x)$, and by \eqref{bound_unif_k}, we deduce
\begin{equation}\label{aia}
    \begin{split}
        &\int_\Omega(\varepsilon +|D\boldsymbol{u}_\varepsilon|^2)^{\frac{p-2-\alpha}{2}} \|D^2\boldsymbol{u}_\varepsilon\|^2 \, dx \\ & \leq \mathcal{C}\left(\int_\Omega{(\varepsilon +|D\boldsymbol{u}_\varepsilon|^2)^{\frac{p-\alpha}{2}}} \, dx\right.\\
        &\qquad \left. + \sum_{i=1}^n \int_{\Omega} \langle \boldsymbol{f}, \partial_{x_i}\left(\frac{\boldsymbol{u}_{\varepsilon,i}}{(\varepsilon+|D{\boldsymbol u}_\varepsilon|^2)^{\frac{\alpha}{2}}}\right)\rangle\,dx\right.\\
        & \qquad\left. +\int_\Omega |\boldsymbol f| |D{\boldsymbol u_\varepsilon}|^{1-\alpha} \,dx \right).
    \end{split}
\end{equation}
\begin{comment}
    &  \qquad\left.+\sum_{i=1}^n\int_\Omega \langle \partial_{x_i} \boldsymbol{f},\boldsymbol{u}_{\varepsilon,i} \rangle \frac{1}{(\varepsilon+|D{\boldsymbol u}_\varepsilon|^2)^{\frac{\alpha}{2}}}\,dx\right.\\
    +\overline{C}\int_{\partial \Omega}  |\boldsymbol{f}|\,d\mathcal{H}^{n-1}
\end{comment}
%\qquad\left.+\sum_{i=1}^n\int_\Omega \langle -\boldsymbol{f}, \partial _{x_i}\left(\frac{\boldsymbol{u}_{\varepsilon,i}}{(\varepsilon+|D{\boldsymbol u}_\varepsilon|^2)^{\frac{\alpha}{2}}}\right)\rangle\,dx
For $\varepsilon\in (0,1)$ fixed, let us define the function $\boldsymbol{g}_\varepsilon \in W^{1,2}(\Omega) \cap L^\infty(\Omega)$:
\begin{equation*}
    \boldsymbol{g}_\varepsilon = \frac{\boldsymbol{u}_{\varepsilon,i}}{(\varepsilon+|D{\boldsymbol u}_\varepsilon|^2)^{\frac{\alpha}{2}}}
\end{equation*}
Again by \cite{Bur}, there exists a sequence $\{\boldsymbol{g}_s\} \subset C^\infty(\Omega) \cap C(\overline \Omega)$  such that $\boldsymbol{g}_s \rightarrow \boldsymbol{g}_\varepsilon$ as $s \rightarrow \infty$. Moreover,
\begin{equation}\label{g_linf}
    \|\boldsymbol{g}_s\|_{L^\infty(\Omega)}\leq C \|\boldsymbol{g}_\varepsilon\|_{L^\infty(\Omega)},
\end{equation}
for some constant $C$ independent on $s$.

To proceed, we treat the second term in the right hand side of \eqref{aia} separately:
\begin{equation*}
    G:=\sum_{i=1}^n \int_{\Omega} \langle \boldsymbol{f}, \partial_{x_i}\left(\frac{\boldsymbol{u}_{\varepsilon,i}}{(\varepsilon+|D{\boldsymbol u}_\varepsilon|^2)^{\frac{\alpha}{2}}}\right)\rangle\,dx = \lim_{s \rightarrow\infty} \sum_{i=1}^n \int_{\Omega} \langle \boldsymbol{f}, \partial_{x_i} \boldsymbol{g}_s\rangle\,dx,
\end{equation*}
where in the previous computation we used the dominated convergence theorem. Integrating by parts, it follows that
\begin{equation*}
\begin{split}
    \lim_{s \rightarrow\infty} \sum_{i=1}^n \int_{\Omega} \langle \boldsymbol{f}, \partial_{x_i} \boldsymbol{g}_s\rangle\,dx=
    &-\lim_{s \rightarrow\infty} \sum_{i=1}^n \int_{\Omega} \langle\partial_{x_i} \boldsymbol{f},  \boldsymbol{g}_s\rangle\,dx\\
    &+ \lim_{s \rightarrow\infty}\sum_{i=1}^n \int_{\partial \Omega} \langle \boldsymbol{f}, \boldsymbol{g}_s \rangle \nu_i \,d\mathcal{H}^{n-1}.
\end{split}
\end{equation*}
Thus, by \eqref{linfinito} and \eqref{g_linf},
\begin{equation}\label{bound_G}
\begin{split}
    |G| &\leq \sum_{i=1}^n \int_{\Omega} |\langle\partial_{x_i} \boldsymbol{f},  \boldsymbol{g}_\varepsilon \rangle|\,dx + \tilde{C} \int_{\partial\Omega} |\boldsymbol{f}| \,d\mathcal{H}^{n-1}\\
    &\leq \tilde{C} \left (\int_{\Omega} |D\boldsymbol{f}|\, dx+ \int_{\partial\Omega} |\boldsymbol{f}| \,d\mathcal{H}^{n-1} \right),
\end{split}
\end{equation}
with $\tilde{C}=\tilde{C}(p,n,|\Omega|,\|\operatorname{tr}\mathcal{B}\|_{L^{n-1,1}(\Omega)},\|\boldsymbol f\|_{L^q(\Omega)},\alpha)$.

By \eqref{bound_G} and a trace inequality in Lipschitz domains, see \cite{Nec}, we can rewrite \eqref{aia} as follows:
\begin{equation}\label{aia2}
    \begin{split}
        &\int_\Omega(\varepsilon +|D\boldsymbol{u}_\varepsilon|^2)^{\frac{p-2-\alpha}{2}} \|D^2\boldsymbol{u}_\varepsilon\|^2 \, dx \\ 
        & \leq \mathcal{C}\left(\int_\Omega{(\varepsilon +|D\boldsymbol{u}_\varepsilon|^2)^{\frac{p-\alpha}{2}}} \, dx\right. \\
        &\qquad \left. + \|\boldsymbol f\|_{L^1(\Omega)} +  \|\boldsymbol f\|_{W^{1,1}(\Omega)} \right),
    \end{split}
\end{equation}
where $\mathcal{C}=\mathcal{C}(n,N,p,\alpha,L_\Omega,d_\Omega,K_{\Omega},\|\boldsymbol f\|_{L^q(\Omega)},|\Omega|,\|\operatorname{tr}\mathcal{B}\|_{L^{n-1,1}(\Omega)})$ is a positive constant.

We notice that, for $\gamma \geq \frac{p-\alpha}{2}$, by \eqref{aia2} and \eqref{linfinito} we arrive to
\begin{equation}\label{second_deriv_eps}
    \int_{\Omega} (\varepsilon + |D \boldsymbol u_\varepsilon |^2)^{\gamma-1} \|D^2 \boldsymbol u_\varepsilon \|^2 \leq \mathcal{C},
\end{equation}
where $\mathcal{C}=\mathcal{C}(n,N,p,\alpha,L_\Omega,d_\Omega,K_{\Omega},|\Omega|,\|\operatorname{tr}\mathcal{B}\|_{L^{n-1,1}(\Omega)},\|\boldsymbol f\|_{L^q(\Omega)},\|\boldsymbol f\|_{W^{1,1}(\Omega)})$.

\vspace{0.5 cm}

A direct computation shows that
\begin{equation}\label{W12ineq}
    \left |\nabla ((\varepsilon +|D\boldsymbol{u}_\varepsilon|^2)^{\frac{\gamma-1}{2}} u_{\varepsilon,x_i}^\beta) \right | \leq C (\varepsilon +|D\boldsymbol{u}_\varepsilon|^2)^{\frac{\gamma-1}{2}} \| D^2 \boldsymbol u_\varepsilon \| \quad \text{a.e. in } \Omega,
\end{equation}
for $i=1,...,n$, $\beta=1,...,N$ and for some positive constant $C=C(p,\alpha,n,N)$.

By \eqref{W12ineq} and \eqref{linfinito}, one can infer
\begin{equation}\label{bound_quasi_stress_field}
\left \|(\varepsilon +|D\boldsymbol{u}_\varepsilon|^2)^{\frac{\gamma-1}{2}} D \boldsymbol u_\varepsilon \right \|_{W^{1,2}(\Omega)}\leq C,
\end{equation}
where $C=C(n,N,p,\alpha,L_\Omega,d_\Omega,\Psi_{\Omega},|\Omega|,\|\operatorname{tr}\mathcal{B}\|_{L^{n-1,1}(\Omega)},\|\boldsymbol f\|_{L^q(\Omega)},\|\boldsymbol f\|_{W^{1,1}(\Omega)})$ is a positive constant. Notice that the dependence of the constant by $\Psi_\Omega$ is due to Lemma \ref{antoninho2}.
We remark that the function $\Psi_\Omega$ has finite value by Remark \ref{antony2}.

The uniform boundedness in $W^{1,2}(\Omega)$ implies that there exists a function $\boldsymbol V \in W^{1,2}(\Omega)$ such that, up to a subsequence,
\begin{equation}\label{weakconv}
    (\varepsilon +|D\boldsymbol{u}_\varepsilon|^2)^{\frac{\gamma-1}{2}} D \boldsymbol u_\varepsilon \rightarrow \boldsymbol V \text{ in } L^2(\Omega), \quad (\varepsilon +|D\boldsymbol{u}_\varepsilon|^2)^{\frac{\gamma-1}{2}} D \boldsymbol u_\varepsilon \rightharpoonup \boldsymbol V \text{ in } W^{1,2}(\Omega),
\end{equation}
as $\varepsilon \rightarrow 0$.\\
Moreover, by \cite[Equation $(5.41)$]{Cma},
  \begin{equation}\label{Lpstress}
      D \boldsymbol u_\varepsilon\rightarrow D\boldsymbol u \quad\text{in } L^p(\Omega).
  \end{equation}
From equations \eqref{weakconv} and \eqref{Lpstress}, we deduce
\begin{equation}\label{stress_C2}
    |D \boldsymbol u|^{\gamma-1} D \boldsymbol u = \boldsymbol V \in W^{1,2}(\Omega),
\end{equation}
for $\gamma \geq \frac{p-\alpha}{2}$.

\textit{Step $2$.} Here, we remove assumption \eqref{bordoregolare}.\\
We consider a sequence of open sets $\{\Omega_m\}$ approximating $\Omega$ in the sense of Lemma \ref{approxcap}. For $m\in \N$, let $\boldsymbol{u}_m$ be the weak solution to the Dirichlet problem 
\beq\label{system34}\tag{$p$-$D_m$}
\begin{cases}
-\operatorname{\bf div}({|D{\boldsymbol u_m}|}^{p-2}D{\boldsymbol u_m})=   {\boldsymbol f}_m(x)& \mbox{in $\Omega_m$}\\
\boldsymbol u_m= 0  & \mbox{on  $\partial \Omega_m$},
\end{cases}
\eeq

where $\boldsymbol f_m\in C^{\infty}_c(\R^n;\R^N)$. Here, $\boldsymbol{f}_m$
  is constructed by extending $\boldsymbol{f}$ to the whole $\R^n$, since $f$ belongs to the Sobolev space, and then using classical mollifiers and standard cut-off functions. With this choice $\boldsymbol f_m$ restricted to $\Omega$ converges to $\boldsymbol f$ in the norm $W^{1,1}(\Omega)\cap L^q(\Omega).$ Moreover $$\|\boldsymbol{f}_m\|_{W^{1,1}(\Omega_m)}\leq C\|\boldsymbol{f}\|_{W^{1,1}(\Omega)}\quad\text{and}\quad \|\boldsymbol{f}_m\|_{L^{q}(\Omega_m)}\leq C \|\boldsymbol{f}\|_{L^{q}(\Omega)},$$
where $C$ is a positive constant independent on $m$.
By inequality \eqref{bound_quasi_stress_field}, applied to $\boldsymbol{u}_m$, for $\gamma \geq \frac{p-\alpha}{2}$, and \eqref{stress_C2} we get
\begin{equation}\label{stress_field_u_m}
\left \||D\boldsymbol{u}_m|^{\gamma-1} D \boldsymbol u_m \right \|_{W^{1,2}(\Omega)}\leq \left \||D\boldsymbol{u}_m|^{\gamma-1} D \boldsymbol u_m \right \|_{W^{1,2}(\Omega_m)} \leq C,
\end{equation}
where $C(n,N,p,\alpha,L_\Omega,d_\Omega,\Psi_{\Omega},\|\operatorname{tr}\mathcal{B}\|_{L^{n-1,1}(\Omega)}\| \boldsymbol f\|_{L^q(\Omega)},\|\boldsymbol f\|_{W^{1,1}(\Omega)})$ is a positive constant. Note that this dependence of the constant $\mathcal{C}$ is guaranteed by proprieties \eqref{prop1} and \eqref{prop2} of the sequence $\{\Omega_m\}$, by the convergence of $\boldsymbol f_m$ to $\boldsymbol{f}$, and by \cite[Step $5$]{Cma2}.

By \eqref{stress_field_u_m}, we deduce that there exists $\boldsymbol{V} \in W^{1,2}(\Omega)$ such that, up to a subsequence 
\begin{equation}\label{weakconv_m}
    |D\boldsymbol{u}_m|^{\gamma-1} D \boldsymbol u_m \rightarrow \boldsymbol V \text{ in } L^2(\Omega), \quad |D\boldsymbol{u}_m|^{\gamma-1} D \boldsymbol u_m \rightharpoonup \boldsymbol V \text{ in } W^{1,2}(\Omega),
\end{equation}

Next, we claim that, for every open set $\tilde \Omega\subset\subset\Omega$, there exist a constant $C$, independent of $m$, such that 
\begin{equation}\label{Farinatop2s}
    \|\boldsymbol{u}_m\|_{C^{1,\alpha}(\tilde \Omega)}\leq C.
\end{equation}

By \cite[Corollary 5.5]{DiKaSc}, we deduce that 
\begin{equation}\label{fatto1s}
    \|D\boldsymbol{u}_m\|_{C^{0,\alpha}(\tilde\Omega)}\leq C(\|\boldsymbol{
    f}_m\|_{C^{0,\alpha}(\tilde \Omega)}),
\end{equation}
where $C(\|\boldsymbol{
    f}_m\|_{C^{0,\alpha}(\tilde \Omega)})$ is a positive constant independent of $m$. Using $\boldsymbol{u}_m$ as test function in  \eqref{system34}, and the regularity of $\boldsymbol{f}_m,\boldsymbol{f} \in L^{(p^*)'}$, where $(p^*)'$ denotes the conjugate exponent of $p^*$, yields 
\begin{equation}\label{uniformemLp}
    \|D\boldsymbol{u}_m\|_{L^p(\Omega_m)}\leq C,
\end{equation}
  for some constant $C$ independent of $m$. Thus, by the Poincaré inequality, 
  \begin{equation}\label{fatto2s}
    \|\boldsymbol{u}_m\|_{L^p(\Omega_m)}\leq C,
\end{equation}
where $C$ is a positive constant independent of $m$. By \eqref{fatto1s} and \eqref{fatto2s}, using a Sobolev type inequality, we have 
\begin{equation*}
    \|\boldsymbol{u}_m\|_{L^{\infty}(\tilde{\Omega})}\leq C,
\end{equation*}
for some constant $C$ independent of $m$. The claim follows by $C^{1,\alpha}$ regularity, see \cite{ChenDiB,DP}.

Therefore, by \eqref{Farinatop2s}, up to subsequence
\begin{equation}\label{strongW12s}
    \boldsymbol{u}_m\rightarrow \boldsymbol{v} \quad \text{in } C^{1,\alpha}(\tilde\Omega).
\end{equation}
In particular, by \eqref{weakconv_m} and using a diagonal procedure,
\begin{equation}\label{stress_strong}
    |D\boldsymbol{u}_m|^{\gamma-1} D \boldsymbol u_m \rightarrow |D\boldsymbol{v}|^{\gamma-1} D \boldsymbol v \quad \text{in } \Omega.
\end{equation}

Now, take as test function $\boldsymbol \varphi \in C_c^\infty(\Omega)$ (extended by $0$ outside $\Omega$) in the weak formulation of \eqref{system34}, and pass to the limit as $m \rightarrow \infty$ in the resulting equation, obtaining
\begin{equation}\label{limit_Dvs}
    \int_{\Omega}|D \boldsymbol{v}|^{p-2}(D\boldsymbol{v}:D\boldsymbol{\varphi})\,dx=\int_\Omega \langle\boldsymbol{f},\boldsymbol{\varphi}\rangle\,dx.
\end{equation}
By \eqref{fatto2s}, we can ensure that $|D \boldsymbol{v}|^{p-2} D\boldsymbol{v} \in L^{p'}(\Omega)$. Thus, by a density argument, equation \eqref{limit_Dvs} holds for every $\boldsymbol \varphi \in W^{1,p}_0(\Omega)$. This implies that $\boldsymbol v$ is a weak solution of \eqref{system1} + \eqref{Dir_cond} and by its uniqueness $\boldsymbol v=\boldsymbol u$.

\textit{Step $3$.} The last step consists in removing assumption \eqref{reg_f}. Let $\boldsymbol{f} \in W^{1,1}(\Omega) \cap L^q(\Omega)$, with $q > n$. By standard density argument one can infer that there exists a sequence $\{\boldsymbol{f}_k\} \subset C^\infty(\overline\Omega)$ (as in Step 2) such that
\begin{equation}\label{conv_f}
    \boldsymbol{f}_k \rightarrow \boldsymbol{f} \quad \text{in } W^{1,1}(\Omega) \cap L^q(\Omega).
\end{equation}

To proceed we consider a sequence $\{\boldsymbol{u}_k\}$ of weak solutions to the following system
\beq\label{system_step3}
\begin{cases}
-\operatorname{\bf div}({|D{\boldsymbol u_k}|}^{p-2}D{\boldsymbol u_k})=   {\boldsymbol f_k}& \mbox{in $\Omega$}\\
\boldsymbol u_k= 0  & \mbox{on  $\partial \Omega$}.
\end{cases}
\eeq

By inequality \eqref{stress_field_u_m} of \textit{Step $2$} applied to $\boldsymbol u_k$ we deduce
\begin{equation}\label{stress_field_u_k}
    \begin{split}
&\left \||D\boldsymbol{u}_k|^{\gamma-1} D \boldsymbol u_k \right \|_{W^{1,2}(\Omega)}\\
&\quad\leq C(n,N,p,\alpha,L_\Omega,d_\Omega,\Psi_{\Omega},\|\operatorname{tr}\mathcal{B}\|_{L^{n-1,1}(\Omega)}\| \boldsymbol f_k\|_{L^q(\Omega)},\|\boldsymbol f_k\|_{W^{1,1}(\Omega)})\\
&\quad\leq \bar C(n,N,p,\alpha,L_\Omega,d_\Omega,\Psi_{\Omega},\|\operatorname{tr}\mathcal{B}\|_{L^{n-1,1}(\Omega)}\| \boldsymbol f\|_{L^q(\Omega)},\|\boldsymbol f\|_{W^{1,1}(\Omega)}),
\end{split}
\end{equation}
where the last inequality follows by \eqref{conv_f}.
\begin{comment}
    Thus, $\exists V \in W^{1,2}(\Omega)$ such that
\begin{equation*}
    |D\boldsymbol{u}_k|^{\gamma-1} D \boldsymbol u_k  \rightharpoonup V \quad \text{in } W^{1,2}(\Omega)
\end{equation*}
Let us take as test function $\boldsymbol \psi \in C^\infty_0(\Omega)$ and consider the weak formulation associated to \eqref{system_step3}:
\begin{equation*}
    \int_{\Omega}|D \boldsymbol{u}_k|^{p-2}(D\boldsymbol{u}_k:D\boldsymbol{\psi})\,dx=\int_\Omega \langle\boldsymbol{f}_k,\boldsymbol{\psi}\rangle\,dx.
\end{equation*}
\end{comment}

We claim that:
\begin{equation}\label{convergenzalp}
    D\boldsymbol{u}_k \rightarrow D\boldsymbol{u} \quad \text{in} \quad L^p(\Omega).
\end{equation}
We note that a basic energy estimate yields 
\begin{equation}\label{Lplimitato}
    \|D\boldsymbol{u}_k\|_{L^p(\Omega)} \leq C,
\end{equation}
with $C$ not depending on $k$.
Now, testing \eqref{system_step3} and \eqref{system1} + \eqref{Dir_cond} with $\boldsymbol{u}_k -\boldsymbol{u}$ and subtracting the resulting terms, we obtain:
\begin{equation}\label{bdd_uk}
\begin{split}
    \int_\Omega &(|D\boldsymbol{u}_k|^{p-2}+|D\boldsymbol{u}|^{p-2})|D \boldsymbol{u}_k-D\boldsymbol{u}|^2 \, dx\\
    &\leq C(n,p)\left ( \int_\Omega (|\boldsymbol{u}_k|^{p*}+|\boldsymbol{u}|^{p*})\,dx \right)^{\frac{1}{p*}} \left ( \int_\Omega|\boldsymbol{f}_k-\boldsymbol{f}|^{(p*)'} \,dx\right)^{\frac{1}{(p*)'}}
\end{split}
\end{equation}

By \eqref{bdd_uk} and \eqref{Lplimitato}, the claim \eqref{convergenzalp} follows (see for details \cite[Section 4]{Cma}).

Summing up, using \eqref{stress_field_u_k} and by claim \eqref{convergenzalp}, we have
\begin{equation*}
    |Du|^{\gamma-1}Du \in W^{1,2}(\Omega).
\end{equation*}

\vspace{0.5 cm}

\textbf{Neumann case.} Similarly, we can prove the same results for the Neumann problem. In what follows, we mention the main changes required in each step.

\textit{Step $1.$} In this framework we take $\boldsymbol u_\varepsilon$, normalized by its mean, as the weak solution of the following Neumann system:
\beq\label{system33N}\tag{$p$-$N_\varepsilon$}
\begin{cases}
-\operatorname{\bf div}((\varepsilon+|D{\boldsymbol u_\varepsilon}|^2)^\frac{p-2}{2}D{\boldsymbol u_\varepsilon})=   {\boldsymbol f}(x)& \mbox{in $\Omega$}\\
{\frac{\partial \boldsymbol u_\varepsilon}{\partial \boldsymbol \nu}} = 0  & \mbox{on  $\partial \Omega$}.
\end{cases}
\eeq
We recover the estimate \eqref{linfinito} by \cite[Theorem $2.4$]{Cma2}. As shown in the proof of that theorem, $\boldsymbol{u_\varepsilon}\in W^{1,\infty}(\Omega)\cap W^{2,2}(\Omega)$, and there exists a sequence $\{\boldsymbol u_k\}\in C^{\infty}(\Omega)\cap C^2(\overline \Omega)$, such that:
\begin{equation*}
\begin{split}
    &\frac{\partial \boldsymbol u_k}{\partial \boldsymbol \nu }= 0 \quad \text{on } \partial \Omega,\\
    &\boldsymbol{u}_k\rightarrow\boldsymbol{u}_\varepsilon \quad \text{in } W^{2,2}(\Omega), \quad D\boldsymbol u_k\rightarrow D\boldsymbol u_\varepsilon \quad \text{a.e. in }\Omega,\\
    &\|D\boldsymbol{u}_k\|_{L^{\infty}(\Omega)}\leq C\|D\boldsymbol{u}_\varepsilon\|_{L^{\infty}(\Omega)},
\end{split}
\end{equation*}
for some constant $C$ independent on $k$. From Theorem \ref{stimaC^2}, using estimate \eqref{abcdef} with $\boldsymbol{u}=\boldsymbol{u}_k$, we get

\begin{equation}\label{abcdefghh}\begin{split}
&\int_\Omega(\varepsilon +|D\boldsymbol{u}_k|^2)^{\frac{p-2-\alpha}{2}} \|D^2\boldsymbol{u}_k\|^2 \, dx \\ & \leq \mathcal{C}\left(\int_\Omega{(\varepsilon +|D\boldsymbol{u}_k|^2)^{\frac{p-\alpha}{2}}} \, dx\right.  \\&  \qquad\left.+\sum_{i=1}^n\int_\Omega \langle \operatorname{\bf div}( (\varepsilon+|D{\boldsymbol u}_k|^2)^{\frac{p-2}{2}}D{\boldsymbol u}_k), \partial _{x_i}\left(\frac{\boldsymbol{u}_{k,i}}{(\varepsilon+|D{\boldsymbol u}_k|^2)^{\frac{\alpha}{2}}}\right)\rangle\,dx\right).\\
\end{split}
\end{equation}
By dominated convergence theorem and integration by parts, since \eqref{Neu_cond} holds, we deduce
\begin{equation*}
\begin{split}
&\int_\Omega(\varepsilon +|D\boldsymbol{u}_\varepsilon|^2)^{\frac{p-2-\alpha}{2}} \|D^2\boldsymbol{u}_\varepsilon\|^2 \, dx \\ & \leq \mathcal{C}\left(\int_\Omega{(\varepsilon +|D\boldsymbol{u}_\varepsilon|^2)^{\frac{p-\alpha}{2}}} \, dx\right.  \\&  \qquad\left.+\sum_{i=1}^n\int_\Omega \langle\partial_{x_i} \boldsymbol{f},\frac{\boldsymbol{u}_{\varepsilon,i}}{(\varepsilon+|D{\boldsymbol u}_\varepsilon|^2)^{\frac{\alpha}{2}}}\rangle\,dx\right).
\end{split}
\end{equation*}

With these variants, the remaining part follows analogously as in the Dirichlet case.

\textit{Step $2$.} Instead of \eqref{system34} we consider
\beq\label{system34N}\tag{$p$-$N_m$}
\begin{cases}
-\operatorname{\bf div}(|D{\boldsymbol u_m}|^{p-2}D{\boldsymbol u_m})=   {\boldsymbol f}_m(x)& \mbox{in $\Omega_m$}\\
{\frac{\partial \boldsymbol u_m}{\partial \boldsymbol \nu}}= 0  & \mbox{on  $\partial \Omega_m$},
\end{cases}
\eeq
and the corresponding sequence of solutions $\{\boldsymbol u_m\}$ has to be normalized by a suitable sequence of additive constant vectors $\boldsymbol \xi_m = -\int_{\Omega_m} \boldsymbol u_m / |\Omega_m|$. By the same argument, we prove that the normalized sequence $\boldsymbol u_m$ admits a subsequence converging to a function $\boldsymbol v$ with the same properties as in the Dirichlet case.

Now, since any test function $\boldsymbol{\psi}\in W^{1,\infty}(\Omega)$ can be extended to a function in $W^{1,\infty}(\R^n)$, the weak formulation of \eqref{system34N} yields
\begin{equation}\label{weakOmega_m}
    \int_{\Omega_m}|D\boldsymbol{u}_m|^{p-2} (D\boldsymbol{u}_m:D\boldsymbol{\psi})\,dx=\int_{\Omega_m} \langle\boldsymbol{f},\boldsymbol{\psi}\rangle\,dx.
\end{equation}
We notice that the left hand side of \eqref{weakOmega_m} can be written as
\begin{equation}\label{bau}
\begin{split}
&\int_{\Omega_m}|D\boldsymbol{u}_m|^{p-2} (D\boldsymbol{u}_m:D\boldsymbol{\psi})\,dx \\ &=\int_{\Omega}|D\boldsymbol{u}_m|^{p-2} (D\boldsymbol{u}_m:D\boldsymbol{\psi})\,dx+\int_{\Omega\setminus\Omega_m}|D\boldsymbol{u}_m|^{p-2} (D\boldsymbol{u}_m:D\boldsymbol{\psi})\,dx.
\end{split}
\end{equation} 
By the dominated convergence theorem, using \eqref{uniformemLp}, \eqref{strongW12s}, \eqref{bau} and the fact that $|\Omega_m \setminus \Omega| \rightarrow 0$ we deduce
\begin{equation}\label{kappa}
    \int_{\Omega}|D \boldsymbol{v}|^{p-2}(D\boldsymbol{v}:D\boldsymbol{\psi})\,dx=\int_\Omega \langle\boldsymbol{f},\boldsymbol{\psi}\rangle\,dx.
\end{equation}
Since $\Omega$ is a bounded Lipschitz domain, by density, \eqref{kappa} holds for any $\boldsymbol{\psi}\in W^{1,p}(\Omega)$, and therefore $\boldsymbol{v}$ is a weak solution of \eqref{system1} + \eqref{Neu_cond}.

\textit{Step $3$.} The last step is entirely analogous, taking as $\boldsymbol{u}_k$ a sequence of weak solutions of the Neumann system
\begin{equation*}
    \begin{cases}
-\operatorname{\bf div}({|D{\boldsymbol u_k}|}^{p-2}D{\boldsymbol u_k})=   {\boldsymbol f_k}& \mbox{in $\Omega$}\\
\frac{\partial\boldsymbol u_k}{\partial \nu}= 0  & \mbox{on  $\partial \Omega$}.
\end{cases}
\end{equation*}

\vspace{0.5 cm}

\end{proof}

\begin{corollary}\label{u_W2,2}
    Let $\boldsymbol u$, $\Omega$ and $\boldsymbol f$ as in Theorem \ref{teo1INTRO}. If $1<p<3$, then
    \begin{equation*}
        \boldsymbol u \in W^{2,2}(\Omega).
    \end{equation*}
\end{corollary}
\begin{proof}
    The thesis follows by an application of Theorem \ref{teo1INTRO} with $\gamma =1$.
\end{proof}

\vspace{0.3 cm}

We conclude this section analyzing the special case of convex domains. 

\begin{proof}[Proof of Theorem \ref{conv_d}]
    The proof is similar to that of Theorem \ref{teo1INTRO}. The only difference relies on choosing, in \textit{Step $2$}, a sequence $\Omega_m$ of bounded convex open sets approximating $\Omega$ from outside with respect to the Hausdorff distance. Furthermore, conditions \eqref{prop1} are automatically fulfilled with convex domains and condition \eqref{prop2} does not play any role since the constant $\mathcal{C}$ in inequalities \eqref{abcdeh} and \eqref{abcdefgh} is independent of $K_\Omega$.
\end{proof}

\section{Integrability of the inverse of the gradient}\label{integrability_section}

We stress that the  results proved in Section \ref{SOE} hold without any sign assumption on the source term $\boldsymbol{f}$. Providing, in addition, that $f^\beta \geq \tau > 0$ or $f^\beta \leq -\tau < 0$ locally for some $\beta$, we deduce, as a consequence of Theorem \ref{teo1INTRO}, the integrability properties of the inverse of the weight $|D \boldsymbol{u}|^{p-2}$. The next lemma provides a global integral estimate for regular functions and smooth domains:
\begin{lemma}\label{ineq_peso_u}
    Let $\Omega$ be a bounded open set in $\R^n$, with $\partial \Omega \in C^2$.\\
    Then, for some $\beta = 1,...,N$, and for any $\theta,\sigma > 0$, we have:
    \begin{equation}\label{stimaprepeso}
   	\begin{split}
   &\left|\int_{\Omega}\operatorname{ div}((\varepsilon+{|D{\boldsymbol u}|)}^\frac{p-2}{2}{\nabla{u}^\beta})\frac{\psi^2}{(\varepsilon+{|D{\boldsymbol u}|)}^\frac{\sigma}{2}}\,dx\right|
        \\
   		&\leq \int_{\Omega} {(\varepsilon + |D \boldsymbol{u}|^2)}^{\frac{p-1-\sigma}{2}}|\nabla \psi|^2 \,dx + \int_{\Omega} {(\varepsilon + |D \boldsymbol{u}
   			|^2)}^{\frac{p-1-\sigma}{2}} \psi^2 \,dx\\
   		&\qquad+ \theta\sigma \int_{\Omega} \frac{\psi^2 }{(\varepsilon + |D \boldsymbol{u}|^2)^{\frac{\sigma}{2}}}\,dx + \frac{\sigma}{4\theta}\int_{\Omega}{(\varepsilon + |D \boldsymbol{u}|^2)}^{p-2-\frac{\sigma}{2}}\|D\boldsymbol{u}\|^2\psi^2\,dx,
   	\end{split}
   \end{equation}
    for any vector field $\boldsymbol{u} \in C^3(\Omega) \cap C^2(\bar \Omega)$ that satisfies either \eqref{Dir_cond} or \eqref{Neu_cond} on $\partial \Omega$ and for any $\psi\in C^{\infty}_c(\R^n)$.
\end{lemma}
\begin{proof}
    We prove inequality \eqref{stimaprepeso} with the negative sign, the other case is similar. Let $\varepsilon \in (0,1)$. Since $u \in C^2(\Omega)$ we are allowed to write:
    \begin{equation*}
        -\operatorname{\bf div}((\varepsilon+{|D{\boldsymbol u}|)}^\frac{p-2}{2}{D \boldsymbol{u}}).
    \end{equation*}
    The idea is to multiply the previous term by a test function $(0,...,\phi^\beta,...,0)$, namely:
    \begin{equation*}
        \phi^\beta = \frac{\psi^2}{{(\varepsilon+|D \boldsymbol{u}|^2)}^{\frac{\sigma}{2}}},
    \end{equation*}
    where $\psi \in C_c^\infty(\mathbb{R}^n)$ and integrate over the domain $\Omega$, namely to consider the following integral:
    \begin{equation*}
        \int_{\Omega}\operatorname{ div}((\varepsilon+{|D{\boldsymbol u}|)}^\frac{p-2}{2}{\nabla{u}^\beta})\frac{\psi^2}{(\varepsilon+{|D{\boldsymbol u}^2|)}^\frac{\sigma}{2}}\,dx.
    \end{equation*}
    Integrating by parts, we arrive to
    \begin{equation*}
    \begin{split}
        -\int_{\Omega}&\operatorname{ div}((\varepsilon+{|D{\boldsymbol u}|)}^\frac{p-2}{2}{\nabla{u}^\beta})\frac{\psi^2}{(\varepsilon+{|D{\boldsymbol u}^2|)}^\frac{\sigma}{2}}\,dx\\
        &=\int_{\Omega} {(\varepsilon + |D\boldsymbol{u}|^2)}^{\frac{p-2}{2}}\langle \nabla u^\beta,\nabla \phi^\beta \rangle\,dx - \int_{\partial \Omega} {(\varepsilon + |D \boldsymbol{u}|^2)}^{\frac{p-2}{2}}\langle \nabla u^\beta,\boldsymbol{\nu}\rangle \phi^\beta \,d\mathcal{H}^{n-1},
    \end{split}
    \end{equation*}
    where $\boldsymbol{\nu}$ is the outward unit vector on $\partial \Omega$.

    To treat the boundary term, let us distinguish two cases:\\
    \begin{enumerate}
        \item[1.] If condition \eqref{Neu_cond} is prescribed, then
        \begin{equation*}
            \int_{\partial \Omega} {(\varepsilon + |D \boldsymbol{u}|^2)}^{\frac{p-2}{2}}\langle \nabla u^\beta,\boldsymbol{\nu}\rangle \phi^\beta \,d\mathcal{H}^{n-1} = 0,
        \end{equation*}
        since $\langle \nabla u^\beta,\boldsymbol{\nu} \rangle=0$.\\
        \item[2.] If, instead, condition \eqref{Dir_cond} is in force, we recall that we can write the outward unit vector on $\partial \Omega$ in terms of $\nabla u^\beta$, namely
        \begin{equation*}
            \boldsymbol{\nu}= \frac{\nabla u^\beta}{|\nabla u^\beta|}, \quad\text{if } \quad\nabla u^\beta \neq 0.
        \end{equation*}
        Otherwise, if $\nabla u^\beta =0$, the boundary term vanishes identically.

        Thus,
        \begin{equation*}
                        \int_{\partial \Omega} {(\varepsilon + |D \boldsymbol{u}|^2)}^{\frac{p-2}{2}}\langle \nabla u^\beta,\boldsymbol{\nu}\rangle \phi^\beta \,d\mathcal{H}^{n-1} =\int_{\partial \Omega} {(\varepsilon + |D \boldsymbol{u}|^2)}^{\frac{p-2}{2}}|\nabla u^\beta| \phi^\beta \,d\mathcal{H}^{n-1} \geq 0.
        \end{equation*}
    \end{enumerate}
    
    \vspace{0.3 cm}
    
    Therefore, in both the Dirichlet and the Neumann cases, holds that:
    \begin{equation}\label{disu1}
        -\int_{\Omega}\operatorname{ div}((\varepsilon+{|D{\boldsymbol u}|)}^\frac{p-2}{2}{\nabla{u}^\beta})\phi^\beta\,dx
        \leq\int_{\Omega} {(\varepsilon + |D\boldsymbol{u}|^2)}^{\frac{p-2}{2}}\langle \nabla u^\beta,\nabla \phi^\beta \rangle\,dx
    \end{equation}
    We focus now on estimating the right hand side of the previous equation
     \begin{equation}\label{disu2}
    \begin{split}
        \int_{\Omega} {(\varepsilon + |D\boldsymbol{u}|^2)}^{\frac{p-2}{2}}\langle \nabla u^\beta,\nabla \phi^\beta \rangle\,dx &= \int_{\Omega} {(\varepsilon + |D \boldsymbol{u}|^2)}^{\frac{p-2}{2}}\langle \nabla u^\beta,\nabla \left ( \frac{\psi^2}{{(\varepsilon + |D \boldsymbol{u}|^2)}^{\frac{\sigma}{2}}}\right ) \rangle\,dx\\
        &=2\int_{\Omega} {(\varepsilon + |D \boldsymbol{u}|^2)}^{\frac{p-2-\sigma}{2}}\langle \nabla u^\beta,\nabla \psi \rangle \psi \,dx\\
        & \quad - \sigma \int_{\Omega} {(\varepsilon + |D \boldsymbol{u}|^2)}^{\frac{p-4-\sigma}{2}}\langle \nabla u^\beta,D^2 \boldsymbol{u} D \boldsymbol{u} \rangle \psi^2 \,dx\\
        &\leq 2\int_{\Omega} {(\varepsilon + |D \boldsymbol{u}|^2)}^{\frac{p-1-\sigma}{2}}|\nabla \psi| |\psi| \,dx\\
        & \qquad \sigma \int_{\Omega} {(\varepsilon + |D \boldsymbol{u}|^2)}^{\frac{p-2-\sigma}{2}}\|D^2 \boldsymbol{u} \| \psi^2 \,dx,
    \end{split}
    \end{equation}
    where in the last inequality we used the estimate \eqref{eq:FraCo}.

    By standard Young inequality we can estimate the first term on the right hand side of the previous inequality as follows
    \begin{equation}\label{disu3}
    \begin{split}
        2\int_{\Omega} &{(\varepsilon + |D \boldsymbol{u}|^2)}^{\frac{p-1-\sigma}{2}}|\nabla \psi| |\psi| \,dx \\
        &\leq \int_{\Omega} {(\varepsilon + |D \boldsymbol{u}|^2)}^{\frac{p-1-\sigma}{2}}|\nabla \psi|^2 \,dx + \int_{\Omega} {(\varepsilon + |D \boldsymbol{u}
        |^2)}^{\frac{p-1-\sigma}{2}} \psi^2 \,dx.
    \end{split}
    \end{equation}
    For the second term, by weighted Young inequality, we deduce
    \begin{equation}\label{disu4}
    \begin{split}
        \sigma \int_{\Omega} &{(\varepsilon + |D \boldsymbol{u}|^2)}^{\frac{p-2-\sigma}{2}}\|D^2 \boldsymbol{u} \| \psi^2 \,dx\\
        &\leq \theta\sigma \int_{\Omega} \frac{\psi^2 }{(\varepsilon + |D \boldsymbol{u}|^2)^{\frac{\sigma}{2}}}\,dx + \frac{\sigma}{4\theta}\int_{\Omega}{(\varepsilon + |D \boldsymbol{u}|^2)}^{p-2-\frac{\sigma}{2}}\|D\boldsymbol{u}\|^2\psi^2\,dx.
    \end{split}
    \end{equation}
    Summing up, using \eqref{disu2}, \eqref{disu3} and \eqref{disu4}, the inequality \eqref{disu1} becomes
    \begin{equation*}
        \begin{split}
            -&\int_{\Omega}\operatorname{ div}((\varepsilon+{|D{\boldsymbol u}|)}^\frac{p-2}{2}{\nabla{u}^\beta})\frac{\psi^2}{(\varepsilon+{|D{\boldsymbol u}|)}^\frac{\sigma}{2}}\,dx\\
            &\leq \int_{\Omega} {(\varepsilon + |D \boldsymbol{u}|^2)}^{\frac{p-1-\sigma}{2}}|\nabla \psi|^2 \,dx + \int_{\Omega} {(\varepsilon + |D \boldsymbol{u}
        |^2)}^{\frac{p-1-\sigma}{2}} \psi^2 \,dx\\
        &\qquad+ \theta\sigma \int_{\Omega} \frac{\psi^2 }{(\varepsilon + |D \boldsymbol{u}|^2)^{\frac{\sigma}{2}}}\,dx + \frac{\sigma}{4\theta}\int_{\Omega}{(\varepsilon + |D \boldsymbol{u}|^2)}^{p-2-\frac{\sigma}{2}}\|D\boldsymbol{u}\|^2\psi^2\,dx.
        \end{split}
    \end{equation*}
  
\end{proof}

Based on the previous lemma we are in position to prove Theorem \ref{peso_stima_Intro}. 

\begin{proof}[Proof of Theorem \ref{peso_stima_Intro}]
The thesis formally follows by an application of Lemma \ref{ineq_peso_u} and by Theorem \ref{teo1INTRO}. Although, as in the proof of Theorem \ref{teo1INTRO}, to apply the previous lemma we call for a regularization argument involving the domain $\Omega$ and the source term of the system $\boldsymbol{f}$. To this aim, we divide the proof into several steps.

\textit{Step $1.$} Here, we assume 
\begin{equation}\label{bordoregolare_P}
   \partial \Omega\in C^2,
\end{equation}
\vspace{0.05cm}
\begin{equation}\label{reg_f_P}
    f \in C^\infty(\Omega).
\end{equation}
Given $\varepsilon \in (0,1)$, we consider a weak solution $\boldsymbol{u_\varepsilon}$ to the approximating system
\begin{equation*}
    -\operatorname{\bf div}((\varepsilon+|D{\boldsymbol u_\varepsilon}|^2)^\frac{p-2}{2}D{\boldsymbol u_\varepsilon})=   {\boldsymbol f}(x) \quad \text{in }\Omega,
\end{equation*}
with either homogeneous Dirichlet or Neumann condition on $\partial \Omega$. As in the proof of Theorem \ref{teo1INTRO}, there exists a sequence $\{\boldsymbol u_k\}\in C^{\infty}(\Omega)\cap C^2(\overline \Omega)$, satisfying either $\boldsymbol{u}_k=0$ or $\frac{\partial \boldsymbol{u}_k}{\partial \boldsymbol{\nu}}=0$ on $\partial \Omega$ for $k\in \mathbb{N}$, and: 
\begin{equation*}
\boldsymbol{u}_k\rightarrow\boldsymbol{u}_\varepsilon \quad \text{in } W^{2,2}(\Omega), \quad D\boldsymbol u_k\rightarrow D\boldsymbol u_\varepsilon \quad \text{a.e. in }\Omega,
\end{equation*}
as $k \rightarrow \infty$. Moreover,
\begin{equation}\label{bound_unif_k_P}
\|D\boldsymbol{u}_k\|_{L^{\infty}(\Omega)}\leq C\|D\boldsymbol{u}_\varepsilon\|_{L^{\infty}(\Omega)},
\end{equation}
for some constants $C$ independent on $k$.

By Lemma \ref{ineq_peso_u}, applied to $\boldsymbol{u}=\boldsymbol{u}_k$, we obtain
  \begin{equation*}
	\begin{split}
		-&\int_{\Omega}\operatorname{ div}((\varepsilon+{|D{\boldsymbol u}_k|)}^\frac{p-2}{2}{\nabla{u}_k^\beta})\frac{\psi^2}{(\varepsilon+{|D{\boldsymbol u}_k|)}^\frac{\sigma}{2}}\,dx\\
		&\leq \int_{\Omega} {(\varepsilon + |D \boldsymbol{u}_k|^2)}^{\frac{p-1-\sigma}{2}}|\nabla \psi|^2 \,dx + \int_{\Omega} {(\varepsilon + |D \boldsymbol{u}_k
			|^2)}^{\frac{p-1-\sigma}{2}} \psi^2 \,dx\\
		&\qquad+ \theta\sigma \int_{\Omega} \frac{\psi^2 }{(\varepsilon + |D \boldsymbol{u}_k|^2)^{\frac{\sigma}{2}}}\,dx + \frac{\sigma}{4\theta}\int_{\Omega}{(\varepsilon + |D \boldsymbol{u}_k|^2)}^{p-2-\frac{\sigma}{2}}\|D\boldsymbol{u}_k\|^2\psi^2\,dx,
	\end{split}
\end{equation*}

   for any $\psi\in C^{\infty}_c(\R^n)$.

 In light of the convergence properties of the sequence $\{\boldsymbol{u}_k\}$, we are in position to use the dominated convergence theorem. 
 
 Thus, for $k \rightarrow \infty$, recalling that $-\operatorname{ div}((\varepsilon+{|D{\boldsymbol u}_\varepsilon|)}^\frac{p-2}{2}{\nabla{u}_\varepsilon^\beta})=f^\beta$, we deduce 

\begin{equation*}
	\begin{split}
		&\int_{\Omega}f^\beta\frac{\psi^2}{(\varepsilon+{|D{\boldsymbol u}_\varepsilon|)}^\frac{\sigma}{2}}\,dx\\
		&\leq \int_{\Omega} {(\varepsilon + |D \boldsymbol{u}_\varepsilon|^2)}^{\frac{p-1-\sigma}{2}}|\nabla \psi|^2 \,dx + \int_{\Omega} {(\varepsilon + |D \boldsymbol{u}_\varepsilon
			|^2)}^{\frac{p-1-\sigma}{2}} \psi^2 \,dx\\
		&\qquad+ \theta\sigma \int_{\Omega} \frac{\psi^2 }{(\varepsilon + |D \boldsymbol{u}_\varepsilon|^2)^{\frac{\sigma}{2}}}\,dx + \frac{\sigma}{4\theta}\int_{\Omega}{(\varepsilon + |D \boldsymbol{u}_\varepsilon|^2)}^{p-2-\frac{\sigma}{2}}\|D\boldsymbol{u}_\varepsilon\|^2\psi^2\,dx.
	\end{split}
\end{equation*}

 As in the proof of Theorem $\ref{stimaC^2}$, there exists $r'' \in (0,r')$, hence depending on $L_\Omega$, $d_\Omega$ and $K_\Omega$ such that $\overline \Omega$ admits a finite covering $\{B_{r_k}\}$, with $r''\leq r_k \leq r'$, and a family of functions $\{\psi_{r_k}\}$
such that $\psi_{r_k} \in C_c^{\infty}(B_{r_k})$ and $\{\psi^2_{r_k}\}$ is a partition of unity of $\overline \Omega$ associated with the covering $\{B_{r_k}\}$. Thus, $\sum_k \psi^2_{r_k}=1$ in $\overline \Omega$. Moreover, the functions $\psi_{r_k}$ can also be chosen so that $|\nabla \psi_{r_k}|\leq C/r_k\leq C / r''$.
 It is crucial to note here that, by our sign assumption on the source term, we can assume that $f^{\beta_k}>\tau_k>0$, in $B_{r_k}$ for some $\beta=1,...,N$.

Writing the previous estimate with $\psi = \psi_k$ for each $k$, using the sign assumptions on $f^\beta$, and adding the resulting inequalities, we deduce
\begin{equation*}
	\begin{split}
		&\int_\Omega \frac{1}{(\varepsilon+{|D{\boldsymbol u}|)}^\frac{\sigma}{2}}\,dx\\
		&\leq C\left(\int_{\Omega} {(\varepsilon + |D \boldsymbol{u}|^2)}^{\frac{p-1-\sigma}{2}} \,dx+ \theta\sigma \int_{\Omega} \frac{1 }{(\varepsilon + |D \boldsymbol{u}|^2)^{\frac{\sigma}{2}}}\,dx\right.\\
		&\qquad +\left. \frac{\sigma}{4\theta}\int_{\Omega}{(\varepsilon + |D \boldsymbol{u}|^2)}^{p-2-\frac{\sigma}{2}}\|D\boldsymbol{u}\|^2\,dx\right),
	\end{split}
\end{equation*}
$C=C(n,L_\Omega,d_\Omega,K_\Omega)$ is a positive constant..

    For $\theta=\theta(\tau_k)$ sufficiently small, we rewrite the previous inequality as follows
    \begin{equation}\label{peso_eps}
    \begin{split}
         \int_\Omega \frac{1}{(\varepsilon + |D \boldsymbol{u}_\varepsilon|^2)^{\frac{\sigma}{2}}}\,dx \leq C &\left (\int_{\Omega} {(\varepsilon + |D \boldsymbol{u}_\varepsilon|^2)}^{\frac{p-1-\sigma}{2}} \,dx \right.\\
         &\quad\left.+\frac{\sigma}{4\theta}\int_{\Omega}{(\varepsilon + |D \boldsymbol{u}_\varepsilon|^2)}^{p-2-\frac{\sigma}{2}}\|D\boldsymbol{u}_\varepsilon\|^2\,dx \right),
    \end{split}
    \end{equation}
      with $C=C(\sigma,\tau,n,L_\Omega,d_\Omega,K_\Omega)$ positive constant.

    We recall that by \cite{Cma2},
\begin{equation}\label{linfinito2}
    \|D\boldsymbol{u}_\varepsilon\|_{L^{\infty}(\Omega)}\leq C(p,|\Omega|,\|\operatorname{tr}\mathcal{B}\|_{L^{n-1,1}(\Omega)},\|\boldsymbol f\|_{L^q(\Omega)}).
\end{equation}

In addition, since $\sigma$ satisfies \eqref{cond_sigma_I}, by \eqref{second_deriv_eps} we infer
\begin{equation}\label{cosette}
    \int_{\Omega}{(\varepsilon + |D \boldsymbol{u}_\varepsilon|^2)}^{p-2-\frac{\sigma}{2}}\|D\boldsymbol{u}_\varepsilon\|^2\,dx \leq C,
\end{equation}
where $C=C(n,N,p,\sigma,L_\Omega,d_\Omega,K_{\Omega},|\Omega|,\|\operatorname{tr}\mathcal{B}\|_{L^{n-1,1}(\Omega)},\|\boldsymbol f\|_{L^q(\Omega)},\|\boldsymbol f\|_{W^{1,1}(\Omega)})$.

Thus, by \eqref{linfinito2} and \eqref{cosette}, we obtain
\begin{equation*}
    \int_\Omega \frac{1}{(\varepsilon + |D \boldsymbol{u}_\varepsilon|^2)^{\frac{\sigma}{2}}}\,dx\leq C,
\end{equation*}
where $C=C(n,N,p,\sigma,\tau,L_\Omega,d_\Omega,\Psi_{\Omega},\|\operatorname{tr}\mathcal{B}\|_{L^{n-1,1}(\Omega)},\|\boldsymbol f\|_{L^q(\Omega)},\|\boldsymbol f\|_{W^{1,1}(\Omega)})$,
notice that the dependence on $\Psi_\Omega$ is due to Lemma \ref{antoninho2}. We remark that the function $\Psi_\Omega$ is valued finite due to Remark \ref{antony2}.

Finally, exploiting Fatou's Lemma, we deduce
\begin{equation}\label{peso_reg2}
\begin{split}
    &\int_\Omega \frac{1}{ |D \boldsymbol{u}|^{\sigma}}\,dx\\ &\leq C(n,N,p,\sigma,\tau,L_\Omega,d_\Omega,\Psi_{\Omega},\|\operatorname{tr}\mathcal{B}\|_{L^{n-1,1}(\Omega)},\|\boldsymbol f\|_{L^q(\Omega)},\|\boldsymbol f\|_{W^{1,1}(\Omega)}),
    \end{split}
\end{equation}
for any $\sigma$ that obeys \eqref{cond_sigma_I}.

\textit{Step $2$.} In order to remove assumption \eqref{bordoregolare_P}, we consider a sequence of open sets $\{\Omega_m\}$ approximating $\Omega$ in the sense of Lemma \ref{approxcap}. For $m\in \N$, let $\boldsymbol{u}_m$ be the weak solution to the problem 
\begin{equation*}
    -\operatorname{\bf div}({|D{\boldsymbol u_m}|}^{p-2}D{\boldsymbol u_m})=   {\boldsymbol f}_m(x) \quad\text{in } \Omega_m,
\end{equation*}
with either $\boldsymbol u_m= 0$ or $\frac{\partial \boldsymbol{u}_m}{\partial \boldsymbol{\nu}}=0$ on $\partial \Omega_m$ and $\boldsymbol{f}_m$ has the same regularity proprieties as in Step 2 of the proof of Theorem \ref{teo1INTRO}.  By inequality \eqref{peso_reg2} of \textit{Step $1$}, applied to $\boldsymbol{u}_m$, 
\begin{equation*}
    \int_\Omega \frac{1}{{|D \boldsymbol{u}_m|}^{\sigma}} \,dx\leq\int_{\Omega_m} \frac{1}{{|D \boldsymbol{u}_m|}^{\sigma}} \,dx \leq C,
\end{equation*}
where $C=C(n,N,p,\sigma,\tau,L_\Omega,d_\Omega,\Psi_{\Omega},\|\operatorname{tr}\mathcal{B}\|_{L^{n-1,1}(\Omega)},\|\boldsymbol f\|_{L^q(\Omega)},\|\boldsymbol f\|_{W^{1,1}(\Omega)})$ is a positive constant. We observe that, due to the properties of $\Omega_m$ (see \cite{Ant}), the chosen covering for $\Omega$ in Step 1 remains the same for 
$\Omega_m$. Moreover, for $m$ sufficiently large, the functions $\boldsymbol{f}_m$ preserve their sign, up to possibly restricting the neighborhood. This follows from the properties of the extension operator and the characteristics of mollifiers. To conclude, we notice that the  dependence of the constant can be obtained as in the proof of Theorem \ref{teo1INTRO}, see \textit{Step $2$}.

Furthermore, by \textit{Step $2$} of the proof of Theorem \ref{teo1INTRO},
\begin{equation*}
    D \boldsymbol{u}_m \rightarrow D \boldsymbol{u} \qquad \text{a.e.} \quad\text{in } \Omega.
\end{equation*}
Thus, Fatou's Lemma ensures:
\begin{equation}\label{peso_reg1}
    \int_\Omega \frac{1}{{|D \boldsymbol{u}|}^{\sigma}}\,dx \leq C.
\end{equation}

\textit{Step $3$.} In the last step we remove the assumption \eqref{reg_f_P}. Let $\boldsymbol{f} \in W^{1,1}(\Omega) \cap L^q(\Omega)$, with $q > n$. By standard density argument one can infer that there exists a sequence $\{\boldsymbol{f}_k\} \subset C^\infty(\overline \Omega)$ (as in Step 2) such that
\begin{equation}\label{conv_f_P}
    \boldsymbol{f}_k \rightarrow \boldsymbol{f} \quad \text{in } W^{1,1}(\Omega) \cap L^q(\Omega).
\end{equation}

To proceed we consider a sequence $\{\boldsymbol{u}_k\}$ of weak solutions to the following system
\begin{equation*}
    -\operatorname{\bf div}({|D{\boldsymbol u_k}|}^{p-2}D{\boldsymbol u_k})=   {\boldsymbol f_k} \quad\text{in } \Omega,
\end{equation*}
with either Dirichlet or Neumann homogeneous boundary conditions.

By inequality \eqref{peso_reg1} of \textit{Step $2$} applied to $\boldsymbol u_k$ we deduce
\begin{equation*}%\label{stress_field_u_k}
    \begin{split}
\int_\Omega \frac{1}{{|D \boldsymbol{u}_k|}^{\sigma}}\,dx &\leq C(n,N,p,\sigma,\tau,L_\Omega,d_\Omega,\Psi_{\Omega},\|\operatorname{tr}\mathcal{B}\|_{L^{n-1,1}(\Omega)},\| \boldsymbol f_k\|_{L^q(\Omega)},\|\boldsymbol f_k\|_{W^{1,1}(\Omega)})\\
&\leq \bar C(n,N,p,\sigma,\tau,L_\Omega,d_\Omega,\Psi_{\Omega},\|\operatorname{tr}\mathcal{B}\|_{L^{n-1,1}(\Omega)},\| \boldsymbol f\|_{L^q(\Omega)},\|\boldsymbol f\|_{W^{1,1}(\Omega)}),
\end{split}
\end{equation*}
where the last inequality follows by \eqref{conv_f_P}.

Moreover, from \textit{Step $3$} of the proof of Theorem \ref{teo1INTRO} we have
\begin{equation*}
    D \boldsymbol{u}_k \rightarrow D \boldsymbol{u} \qquad \text{a.e.} \quad\text{in } \Omega.
\end{equation*}
Hence, applying Fatou's Lemma, the thesis follows.

\end{proof}

%\textcolor{red}{va messa una proof per i convessi e poi sotto nonbisogna fare i vari step o funziona gia con bordo W qualcosa s3econdo me vanno fatti i vari step}
Let us now analyze the special case of convex domains, namely Theorem \ref{conv_d_peso}.

\begin{proof}[Proof of Theorem \ref{conv_d_peso}]
    As in the proof of Theorem \ref{conv_d}, the result follows via Theorem \ref{peso_stima_Intro}, with some changes. The only difference relies on choosing, in \textit{Step $2$}, a sequence $\Omega_m$ of bounded convex open sets approximating $\Omega$ from outside with respect to the Hausdorff distance. We notice that we can use Theorem \ref{conv_d} to deduce the estimate \eqref{cosette} with constant independent of $K_\Omega$. The remaining part follows in the same way as in the proof of Theorem \ref{peso_stima_Intro}.
\end{proof}

\vspace{0.5 cm}

Using the estimates for $D^2 \boldsymbol{u}$ and ${|D\boldsymbol{u}|}^{-1}$ we are able to obtain more regularity for the solution $\boldsymbol{u}$ in the case $p \geq 3$:
\begin{corollary}
     Let $\boldsymbol u$, $\Omega$ and $\boldsymbol f$ as in Theorem \ref{peso_stima_Intro}. If $p\geq 3$, then
     \begin{equation*}
         \boldsymbol{u}\in W^{2,q}(\Omega), \quad \text{with} \quad 1\leq q < \frac{p-1}{p-2}.
     \end{equation*}
\end{corollary}
\begin{proof}
    Given $\varepsilon \in (0,1)$, we consider a weak solution $\boldsymbol{u_\varepsilon}$ to the approximating system
\begin{equation*}
    -\operatorname{\bf div}((\varepsilon+|D{\boldsymbol u_\varepsilon}|^2)^\frac{p-2}{2}D{\boldsymbol u_\varepsilon})=   {\boldsymbol f}(x) \quad \text{in }\Omega,
\end{equation*}
with either homogeneous Dirichlet or Neumann condition on $\partial \Omega$.
    In order to use the estimates proved in the previous theorems, we split ${|D^2 u_\varepsilon|}^q$ as:
    \begin{equation*}
        {|D^2 u_\varepsilon|}^q=|D \boldsymbol{u}_\varepsilon|^{\frac{(p-2-\alpha)q}{2}} {|D^2 u_\varepsilon|}^q \frac{1}{|D \boldsymbol{u}_\varepsilon|^{\frac{(p-2-\alpha)q}{2}}}
    \end{equation*}
    Integrating over $\Omega$ and using the estimate \eqref{second_deriv_eps}, noticing that if $p\geq3$, then $q<2$, we arrive to
    \begin{equation*}
    \begin{split}
        \int_{\Omega} &{|D^2 u_\varepsilon|}^q\,dx = \int_{\Omega}(\varepsilon+ |D \boldsymbol{u}_\varepsilon|^2)^{\frac{(p-2-\alpha)q}{4}} {|D^2 u_\varepsilon|}^q \frac{1}{(\varepsilon+|D \boldsymbol{u}_\varepsilon|^2)^{\frac{(p-2-\alpha)q}{4}}}\,dx\\
        &\leq \left ( \int_{\Omega}(\varepsilon+ |D \boldsymbol{u}_\varepsilon|^2)^{\frac{(
        p-2-\alpha}{2}} {|D^2 u_\varepsilon|}^2 \,dx \right)^{\frac{q}{2}} \left ( \int_\Omega\frac{1}{(\varepsilon+|D \boldsymbol{u}_\varepsilon|^2)^{\frac{(p-2-\alpha)q}{2(2-q)}}}\,dx\right)^{\frac{2-q}{2}}\\
        & \leq C\left ( \int_\Omega\frac{1}{(\varepsilon+|D \boldsymbol{u}_\varepsilon|^2)^{\frac{(p-2-\alpha)q}{2(2-q)}}}\,dx\right)^{\frac{2-q}{2}}
    \end{split}
    \end{equation*}
    We notice that since $p \geq 3$, taking $\alpha \simeq 1$, if $q < \frac{p-1}{p-2}$ the right hand side is bounded. Indeed, it follows by Theorem \ref{peso_stima_Intro}, since $\frac{(p-2-\alpha)q}{2-q} < p-1$.\\
    Therefore,
    \begin{equation*}
        \int_{\Omega} {|D^2 \boldsymbol u_\varepsilon|}^q\,dx \leq C,
    \end{equation*}
where $C$ does not depend on $\varepsilon$. Now it is easy to conclude that $\boldsymbol{u} \in W^{2,q}(\Omega)$. Indeed, by the previous estimate, it follows that there exists $\boldsymbol{V} \in W^{1,q}(\Omega)$ such that
\begin{equation*}
     D\boldsymbol{u}_{\varepsilon} \rightarrow \boldsymbol{V} \text{ in } L^q(\Omega), \quad
           D\boldsymbol{u}_{\varepsilon} \rightharpoonup \boldsymbol{V}\text{ in }W^{1,q}(\Omega)\quad \text{as  } \varepsilon \rightarrow 0. 
\end{equation*}
Moreover, we recall that (see \cite{Cma}):
\begin{equation*}
    D\boldsymbol{u}_{\varepsilon} \rightarrow D\boldsymbol{u} \quad\text{ in } L^p(\Omega),
\end{equation*}
Hence,
\begin{equation*}
    D \boldsymbol{u}=\boldsymbol{V} \in W^{1,q}(\Omega).
\end{equation*}

\end{proof}

\begin{center}
{\bf Acknowledgements}
\end{center} 
The authors are supported by PRIN PNRR P2022YFAJH \emph{Linear and Nonlinear PDEs: New directions and applications.} B. Sciunzi and D. Vuono have been partially supported by \emph{INdAM-GNAMPA Project Regularity and qualitative aspects of nonlinear PDEs via variational and non-variational approaches} E5324001950001.

\

\noindent All the authors are partially supported also by Gruppo Nazionale per l’Analisi Matematica, la Probabilit\`a  e le loro Applicazioni (GNAMPA) of the Istituto Nazionale di Alta Matematica (INdAM).

\begin{center}
	{\sc Data availability statement}\
	All data generated or analyzed during this study are included in this published article.
\end{center}

\

\begin{center}
	{\sc Conflict of interest statement}
	\
	The authors declare that they have no competing interest.
\end{center}

%\
%
%\noindent $\diamond$ The authors would really like to thank the anonymous referee for his/her useful hints and comments.

\end{document}